\documentclass[a4paper,12pt]{article}

\usepackage{amsmath}
\usepackage{amsthm}
\usepackage{amssymb}

\usepackage[english]{babel}

\usepackage{tikz}

\usepackage{comment}
\usepackage{enumitem}

\newcommand{\mbb}{\mathbb}

\theoremstyle{plain}
\newtheorem{thm}{Theorem}[section]
\newtheorem{defi}[thm]{Definition}
\newtheorem{lem}[thm]{Lemma}
\newtheorem{cor}[thm]{Corollary}

\theoremstyle{remark}

\newtheorem*{rem*}{Remark}
\newtheorem*{conj*}{Conjecture}

\numberwithin{equation}{section}

\newcommand{\bsa}{{\boldsymbol{a}}}

\newcommand{\bsd}{{\boldsymbol{d}}}

\newcommand{\NN}{{\mathbb{N}}} 
\newcommand{\RR}{{\mathbb{R}}} 
\newcommand{\ZZ}{{\mathbb{Z}}} 

\DeclareSymbolFont{bbold}{U}{bbold}{m}{n}
\DeclareSymbolFontAlphabet{\mathbbold}{bbold}

\newcommand{\Addresses}{{
  \bigskip
  \footnotesize

	\noindent
  Sigrid Grepstad, \textsc{Department of Mathematical Sciences, Norwegian University of Science and Technology, 7491 Trondheim, Norway.} 
	\par\nopagebreak
	\noindent \textit{E-mail address}: \texttt{sigrid.grepstad@ntnu.no}

  \medskip
	\noindent
  Mario Neum\"uller, \textsc{Department of Financial Mathematics and applied 
	Number Theory, Johannes Kepler University Linz, Altenbergerstraße 69, 4040 
	Linz, Austria}
  \par\nopagebreak
  \noindent \textit{E-mail address}: \texttt{mario.neumueller@jku.at}
}}

\allowdisplaybreaks

\begin{document}

\title{Asymptotic behaviour of the Sudler product of sines for quadratic irrationals}
\author{Sigrid Grepstad and Mario Neum\"uller \thanks{The authors are supported by the
 Austrian Science Fund (FWF): Project F5505-N26, which is part of the Special Research 
Program ``Quasi-Monte Carlo Methods: Theory and Applications''. The work resulting in this paper was partially carried out at the Erwin Schr\"{o}dinger Institute in Vienna, Austria.}}

\date{\today}

\maketitle

\begin{abstract}
We study the asymptotic behaviour of the sequence of sine products $P_n(\alpha) = \prod_{r=1}^n |2\sin \pi r \alpha|$ for real quadratic irrationals $\alpha$. In particular, we study the subsequence $Q_n(\alpha)=\prod_{r=1}^{q_n} |2\sin \pi r \alpha|$, where $q_n$ is the $n$th best approximation denominator of $\alpha$, and show that this subsequence converges to a periodic sequence whose period equals that of the continued fraction expansion of $\alpha$. This verifies a conjecture recently posed by Mestel and Verschueren in \cite{MV15}.
\end{abstract}

\centerline{\begin{minipage}[hc]{130mm}{
{\em Keywords:} Trigonometric product, real quadratic irrational, continued fraction, Kronecker sequence, Lehmer sequence, $q$-series. \\
{\em MSC 2010:} 11J70, 41A60}
\end{minipage}} 

\section{Introduction}
In this paper, we study the sequence of sine products 
$$P_n(\alpha) = \prod_{r=1}^{n} |2 \sin \pi r \alpha|$$
for irrational $\alpha>0$. Early studies of this product were conducted by Erd\H{o}s and Szekeres \cite{erdos} and Sudler \cite{sudler} in the 1960s, and in following decades the sequence has proved important to both pure and applied mathematics (see e.g.\ \cite{bell,borgain,freiman}, or \cite{buslaev,driver,LUB99,petruska} for a connection to $q$-series). It appears that research has been carried out simultaneously, and partly independently, in different mathematical disciplines, resulting in a number of different terminologies and representations of $P_n(\alpha)$. For a brief summary of key results we recommend the introduction of \cite{MV15}.

The focus of this paper will be the subsequence
\begin{equation}
Q_n(\alpha) := \prod_{r=1}	^{q_n} | 2 \sin \pi r \alpha| ,
\label{eq:defqn}
\end{equation}  
where $(q_n)_{n\geq 0}$ are the best approximation denominators of $\alpha$. In a recent paper, Mestel and Verschueren study $Q_n(\alpha)$ in the special case where $\alpha= \omega := (\sqrt{5}-1)/2$ is the fractional part of the golden mean \cite{MV15}. For this case, it was suggested by Knill and Tangerman in \cite{KT11} that the limit value $\lim_{n\rightarrow \infty} Q_n(\omega)$ might exist, and this is confirmed by Mestel and Verschueren.
\begin{thm}[{\cite[Theorem 2.2]{MV15}}]
\label{thm:mvgolden}
If $\omega$ denotes the golden mean and $(F_n)_{n\geq 1}=(1,1,2,3,5,\ldots)$ the Fibonacci sequence, then there exists a constant $c>0$ such that
\begin{equation*}
\lim_{n\rightarrow \infty} Q_n(\omega) = \lim_{n\rightarrow \infty} \prod_{r=1}^{F_n} | 2 \sin \pi r \omega | = c .
\end{equation*}
\end{thm}

Mestel and Verschueren conjecture in \cite{MV15} that Theorem \ref{thm:mvgolden} can be extended to all quadratic irrationals. More precisely, they suggest that if the continued fraction expansion of $\alpha$ has period $\ell$, then the subsequence $Q_n(\alpha)$ will converge to a periodic sequence whose period length divides $\ell$. Our main goal is to verify this claim. 
\begin{thm}
\label{MainThm}
Suppose $\alpha$ has a purely periodic continued fraction expansion $\alpha=[0;\overline{a_1,\ldots,a_{\ell}}]$ with $a_1,\ldots,a_{\ell} \in \NN$ and period $\ell$. Let $(q_n)_{n\geq 1}$ be the sequence of best approximation denominators of $\alpha$. Then there exist positive constants $C_0$, $C_1, \ldots , C_{\ell-1}$ such that
\begin{equation*}
\lim_{m\rightarrow \infty} Q_{\ell m + k}(\alpha) = \lim_{m\rightarrow \infty} \prod_{r=1}^{q_{\ell m+k}} |2 \sin \pi r \alpha| = C_k
\end{equation*}
for each $k=0,1,2, \ldots , \ell -1$.
\end{thm}

\begin{cor}\label{Cor:preperiod}	
Suppose $\beta$ has continued fraction expansion of the form $\beta=[a_0;a_1,\ldots, a_h,\overline{a_{h+1},\ldots,a_{h+\ell}}]$ and let $\alpha=[0;\overline{a_{h+1},\ldots,a_{h+\ell}}]$. We then have
	\begin{equation*}
		\lim_{m\to \infty} Q_{h + \ell m +k}(\beta)=\lim_{m\to \infty}Q_{\ell m+k}(\alpha).
	\end{equation*}
\end{cor}

The proof of Theorem \ref{MainThm} (and Corollary \ref{Cor:preperiod}) largely follows that given by Mestel and Verschueren for the special case of the golden mean. Nevertheless, we include the proof in full detail for the sake of completeness. We emphasize that the challenge in generalizing Theorem \ref{thm:mvgolden} to all quadratic irrationals lies in finding appropriate analogs for $(q_n)_{n\geq 0}$ of certain special properties of the Fibonacci sequence $(F_n)_{n\geq 0}=(0,1,1,2,3,5,8,13,\ldots )$. Throughout their proof for the golden mean case, Mestel and Verschueren make heavy use of the identities
\begin{equation*}
F_n\omega^n = \frac{1}{\sqrt{5}} + \mathcal{O}(\omega^{2n}) \quad  (\text{for } n > 0)
\end{equation*}
and
\begin{equation*}
\frac{F_{n-1}}{F_n} = \omega + \mathcal{O}(\omega^{2n}) \quad (\text{for } n>0), 
\end{equation*}
which do not have obvious analogs for the more general case of a quadratic irrational $\alpha$. However, we will see that similar identities can indeed be formulated for the sequence $(q_n)_{n\geq 0}$ of best approximation denominators of $\alpha$, and with these established the proof of Mestel and Verschueren easily carries over. 

To close this introduction, we give a brief outline of the paper. The existence of $\lim_{m\rightarrow \infty} Q_{\ell m +k}$ claimed by Theorem \ref{MainThm} is verified by splitting the product $Q_{\ell m+k}$ into three more manageable products 
\begin{align*}
Q_{\ell m + k} &= A_mB_mC_m \\
&= \left| 2q_n \sin \pi e_kb^m \right| \cdot \left| \prod_{t=1}^{q_n-1} \frac{s_{mt}}{2\sin(\pi t/q_n)} \right| \cdot \prod_{t=1}^{q_n-1}\left( 1- \frac{s_{m0}^2}{s_{mt}^2}\right)^{1/2},
\end{align*}
where $n=\ell m +k$, $e_k$ is a $k$-dependent constant, and $s_{mt}$ is a perturbed rational sine function to be introduced. This decomposition is explained in detail in Section~\ref{sec:decomp}, where we also show the straightforward convergence of $A_m$ as $m\rightarrow \infty$. The convergence of $B_m$ and $C_m$ is more involved, and is therefore treated in subsequent Sections \ref{sec:Cm} and \ref{sec:Bm}. Prior to this, in Section \ref{sec:propqn}, we establish analogs for $(q_n)_{n\geq 1}$ of the above-mentioned Fibonacci indentities. In particular, we point out a connection to so-called Lehmer sequences, which we consider to be of independent interest (see Theorem \ref{thm:lehmer}). Finally, we summarize the proofs of Theorem \ref{MainThm} and Corollary \ref{Cor:preperiod} in Section \ref{sec:proof}, and gather some concluding remarks on a conjecture of Lubinsky in Section \ref{sec:remarks}. First, however, we introduce necessary notation and some general theory on continued fraction expansions in the following section.


\section{Preliminaries \label{sec:prelim}}

\subsection{Notation \label{sec:notation}}
Throughout the paper, the following notation will be used:

\begin{itemize}
\item We denote by $\{x\}=x-\lfloor x \rfloor$ the fractional part of $x \in \RR$. Moreover, for $t\in \ZZ$ and $m\in \NN$, we denote by $t \bmod m$ the unique remainder of $t/m$ in $\{ 0,1, \ldots , m-1\}$.

\item We use standard big $\mathcal{O}$ notation, and write $f(x) = \mathcal{O}(g(x))$ as $x \rightarrow \infty$ if two functions $f$ and $g$ satisfy $|f(x)|\leq C|g(x)|$ for some $C>0$ and all sufficiently large $x \in \RR$. Moreover, we write $f(x)\sim g(x)$ if $\lim_{x\rightarrow \infty} f(x)/g(x) = 1$.

\item Following Mestel and Verschueren \cite{MV15}, we introduce a generalized sum and product notation: given a summable sequence $(b_r)_{r\in\NN}$, we define the step function $f(t)=b_r$ for $t\in [r, r+1)$. Then for any $x, y \in \RR$ where $x\leq y$, we let 
$$\sum_{r=x}^y b_r := \int_x^y f(t) \, dt.$$
Moreover, if $f(t)>0$ on $[x,y]$, we let 
$$\prod_{r=x}^y b_r := \exp \left( \sum_{r=x}^y \log b_r \right) = \exp \left( \int_x^y \log f(t) \, dt \right). $$ 
This allows us to define sums and products with real, rather than just integer, upper and lower bounds. Note in particular that this definition coincides with normal summation and product notation whenever $x,y \in \ZZ$.
\end{itemize}

\subsection{Permutation operators \label{sec:permutation}}
We will use bold letters to denote vectors. Whenever we have an $\ell$-dimensional, integer-valued vector, for instance $\bsd =(d_1, \ldots , d_{\ell}) \in \NN^{\ell}$, we use the corresponding greek letter (in this case $\delta$) to denote the real number with continued fraction expansion $\delta = [0;\overline{d_1,\ldots,d_{\ell}}]$.

We introduce two families of permutation operators acting on $\NN^{\ell}$: Let $\tau_u : \NN^{\ell} \rightarrow \NN^{\ell}$ be defined by
\begin{equation}
\label{eq:tau}
\tau_u(\bsd) := (d_{u+1}, \ldots , d_{\ell}, d_1, \ldots , d_u) , \quad u \in \{0,1, \ldots , \ell-1 \}, 
\end{equation}
and similarly $\sigma_u : \NN^{\ell} \rightarrow \NN^{\ell}$ be defined by
\begin{equation}
\label{eq:sigma}
\sigma_u(\bsd) := (d_{u-1}, \ldots , d_1, d_{\ell}, \ldots , d_u), \quad u \in \{ 2, 3, \ldots , \ell -1 \},
\end{equation}
with $\sigma_0(\bsd)=(d_{\ell -1}, \ldots , d_1, d_{\ell})$ and $\sigma_1(\bsd)=(d_{\ell}, \ldots , d_1)$. 
Moreover, we use $\delta_{\tau_u}$ and $\delta_{\sigma_u}$ to denote the real numbers with continued fraction expansions given by $\tau_u(\bsd)$ and $\sigma_u(\bsd)$, respectively. That is, we write
\begin{equation}
\label{eq:deltatau}
\delta_{\tau_u}=[0;\overline{d_{u+1}, \ldots , d_{\ell}, d_1, \ldots , d_u}]
\end{equation}
and 
\begin{equation}
\label{eq:deltasigma}
\delta_{\sigma_u}=[0;\overline{d_{u-1}, \ldots , d_1, d_{\ell}, \ldots , d_u}] .
\end{equation}
Our motivation for introducing the operator $\tau_u$ is explained by Lemma \ref{BasicPropQP} in the following subsection. The need to introduce $\sigma_u$ is less evident, but will be clear from Lemma \ref{lem:qnquot} in Section \ref{sec:consequences}, where we describe the asymptotic behaviour of the sequence of denominator quotients $(q_{n-1}/q_n)_{n\geq 1}$ for a quadratic irrational number.
 
\subsection{Continued fraction expansions}
We briefly review some facts about continued fraction expansions of real numbers. In general, for any irrational, real $\alpha \in (0,1)$ whose continued fraction expansion is given by 
$$[0; a_1, a_2, \ldots ] ,$$ 
we denote its $n$th convergent by $p_n/q_n$. The numerators $p_n$ and denominators $q_n$ are given recursively by 
\begin{equation*}
\begin{aligned}
q_0=0, \quad &q_1=1  \quad &q_{n+1}=a_{n}q_n+q_{n-1} ;\\
p_0=1, \quad &p_1=0  \quad &p_{n+1}=a_{n}p_n+p_{n-1} .
\end{aligned}
\end{equation*}
Note that the indexing of $p_n$ and $q_n$ is offset by one compared to what is normally seen in literature. As a consequence, the $n$th convergent $p_n/q_n$ is smaller than $\alpha$ for every odd value of $n$, and greater than $\alpha$ for every even value of $n$. It follows readily from the recurrences above that 
\begin{equation}
\label{eq:pq}
p_nq_{n+1}-p_{n+1}q_n = (-1)^n,
\end{equation}
and as a consequence of this identity we have the error bound
\begin{equation}
\label{eq:stderror}
\left| \alpha - \frac{p_n}{q_n} \right| < \frac{1}{q_{n+1}q_n}
\end{equation}
for the $n$th convergent of $\alpha$. 
\begin{rem*}
Whenever it is \emph{not} clear from context, we write $p_n(\alpha)$ and $q_n(\alpha)$ to indicate that these are the best approximation numerators and denominators corresponding to the real number $\alpha$.
\end{rem*}

\begin{thm}[Ostrowski representation]\label{OstrowskiRepr}
Let $\alpha \in (0,1)$ be an irrational number with continued fraction expansion $[0;a_1, a_2, \ldots]$ and best approximation denominators $(q_n)_{n\geq 1}$. Then every non-negative integer $N$ has a unique expansion
\begin{equation}
\label{eq:Ostrowski}
N=\sum_{n=1}^z v_n q_n ,
\end{equation}
where:
\begin{enumerate}[label=\roman*)]
\item $0\leq v_1< a_1-1$ and $0\leq v_n \leq a_n$ for $n>1$.
\item If $v_n=a_n$ for some $n$, then $v_{n-1}=0$.
\item $z=z(N)=\mathcal{O}(\log(N))$
\end{enumerate}
We refer to \eqref{eq:Ostrowski} as the Ostrowski representation of $N$ in base $\alpha$.
\end{thm}
\noindent The proof of Theorem \ref{OstrowskiRepr} can be found in \cite[p.~126]{KNUnidist}.

\subsection{Periodic continued fraction expansions \label{cfexpansions}}
Suppose now that $\alpha$ is an irrational with $\ell$-periodic continued fraction expansion $\alpha=[0; \overline{a_1, \ldots , a_{\ell}}]$. For this special case, further properties of the convergents $p_n/q_n$ of $\alpha$ can be established. The following lemma summarizes useful relations for $(p_n)_{n\geq0}$ and $(q_n)_{n\geq 0}$ established by Perron in \cite[p.~14--17]{Perron}. 
\begin{lem}\label{BasicPropQP}
Let $\alpha=[0;\overline{a_1,\ldots,a_{\ell}}]$ and for every $k\in\{0,\ldots,\ell-1\}$ let $\tau_k$ be defined as in \eqref{eq:tau}.
	\begin{enumerate}[label=(\alph*)]
		\item \label{BasicPropQPFirst} For all $n, m \in \NN$, we have
			$$p_{n+m}(\alpha)q_{m}(\alpha) - p_{m}(\alpha)q_{n+m}(\alpha)=
			(-1)^{m-1}q_{n}(\alpha_{\tau_{m \bmod{\ell}}}),$$
			where $\alpha_{\tau_k}$ is defined in \eqref{eq:deltatau}.
		\item \label{BasicPropQPSecond} For all $r\in \NN_0$, we have 
			$$q_{\ell+r}(\alpha)=q_{\ell+1}(\alpha)q_{r}(\alpha)+q_{\ell}(\alpha)p_{r}(\alpha)$$ 
			and 
			$$p_{\ell+r}(\alpha)=p_{\ell}(\alpha)p_{r}(\alpha)+p_{\ell+1}(\alpha)q_{r}(\alpha).$$  
		\item \label{BasicPropQPThird} For every $k\in \{1,2, \ldots , \ell-1\}$, we have
			 $$q_{\ell-1}(\alpha_{\tau_k})=p_{\ell}(\alpha_{\tau_{k-1}}).$$
	\end{enumerate}
\end{lem} 

Let us now associate to $\alpha=[0;\overline{a_1, \ldots a_{\ell}}]$ the constant
\begin{equation}
\label{eq:c}
c(\alpha):= q_{\ell+1}(\alpha) + p_{\ell}(\alpha).
\end{equation}
This constant will play an important role as we go forward. As a first application, it appears in the following recursion formula for $(q_n)_{n\geq 0}$.
\begin{lem}\label{qnRecursion}
	Let $\alpha=[0;\overline{a_1,\ldots,a_{\ell}}]$ and let $(q_n)_{n\geq0}$ be the sequence of best approximation denominators of $\alpha$. For all $n\geq2\ell$ we have
	\begin{equation}\label{eq:qnrecur}
	q_n= c(\alpha)q_{n-\ell} + (-1)^{\ell-1}q_{n-2\ell},
	\end{equation}
	with $c(\alpha)$ given in \eqref{eq:c}.
\end{lem}
\begin{proof}[Proof by induction]
For $n=2\ell$, the right hand side in \eqref{eq:qnrecur} reads
$$c(\alpha)q_{\ell} + (-1)^{\ell-1}q_{0}= q_{\ell}q_{\ell+1}+q_{\ell}p_{\ell}.$$
It follows from Lemma \ref{BasicPropQP}\ref{BasicPropQPSecond} with $r=\ell$ that also 
$$q_{2\ell}= q_{\ell}q_{\ell+1}+q_{\ell}p_{\ell},$$
so \eqref{eq:qnrecur} holds for $n=2\ell$. Now let $n=2\ell +1$. The right hand side in \eqref{eq:qnrecur} then reads
$$c(\alpha)q_{\ell+1}+(-1)^{\ell-1}q_1= q_{\ell+1}^2 + p_{\ell}q_{\ell+1}+(-1)^{\ell-1}q_1 = q_{\ell+1}^2+p_{\ell+1}q_{\ell},$$
where the last equality follows from Lemma \ref{BasicPropQP}\ref{BasicPropQPFirst} with $m=\ell$ and $n=1$. Again, using Lemma \ref{BasicPropQP}\ref{BasicPropQPSecond} with $r=\ell+1$, we have 
$$q_{2\ell+1}=q_{\ell+1}^2+p_{\ell+1}q_{\ell},$$ 
so \eqref{eq:qnrecur} holds for $n=2\ell+1$. 

For general $n> 2\ell+1$, we have
\begin{equation*}
q_{n+1}=a_{n}q_n + q_{n-1} = a_{n \bmod \ell} q_n + q_{n-1},
\end{equation*}
where we understand $a_0$ as $a_{\ell}$. Using the induction hypothesis for $q_n$ and $q_{n-1}$, and the fact that $n \bmod \ell = (n-\ell) \bmod \ell = (n-2\ell) \bmod \ell$, we get
\begin{align*}
q_{n+1}&=c(\alpha) \left( a_{n \bmod \ell} q_{n-\ell} + q_{n-\ell-1} \right) + (-1)^{\ell-1} \left( a_{n \bmod \ell} q_{n-2\ell}+q_{n-2\ell-1} \right)\\
&= c(\alpha) q_{n+1-\ell} + (-1)^{\ell-1} q_{n+1-2\ell}.
\end{align*}
This completes the proof of Lemma \ref{qnRecursion}.
\end{proof}

We complete this section by observing that the constant $c(\alpha)$ is, in a sense, independent of the permutation operators $\tau_u$ and $\sigma_u$ introduced in Section \ref{sec:permutation}.
\begin{lem}\label{BasicPropC}
Let $\alpha=[0; \overline{a_1, \ldots, a_{\ell}}]$, and let $c(\alpha)$ be given in \eqref{eq:c}. Moreover, let $\alpha_{\tau_u}$ and $\alpha_{\sigma_u}$ be defined as in \eqref{eq:deltatau} and \eqref{eq:deltasigma}. 
We have
\begin{equation}\label{eq:BasicPropC}
c(\alpha) = c(\alpha_{\tau_u})=c(\alpha_{\sigma_u}),
\end{equation}
for every $u \in \{ 0, 1, \ldots , \ell-1 \}$.
\end{lem}
\begin{proof}
Given an integer vector $\bsd =(d_1, \ldots , d_{\ell}) \in \NN^{\ell}$ and the corresponding irrational number $\delta = [0; \overline{d_1, \ldots , d_{\ell}}]$, we define two sequences of matrices $A_n(\bsd), B_n(\bsd) \in \NN^{n \times n}$, where $A_0(\bsd):=\begin{pmatrix}1\end{pmatrix}$, $A_1(\bsd):=\begin{pmatrix}d_1\end{pmatrix}$, $B_0(\bsd):=\begin{pmatrix}1\end{pmatrix}$, $B_1(\bsd):=\begin{pmatrix}0\end{pmatrix}$ and 
\begin{align*}
		A_n(\bsd)=\begin{pmatrix}
		d^{(\ell)}_1    & -1         & 0         & \dots  & 0   \\   
		1           & d^{(\ell)}_2 & -1         & \ddots & \vdots\\      
		0            & 1        & d^{(\ell)}_3 & \ddots & 0      \\
		\vdots       & \ddots    & \ddots    & \ddots & -1 \\
		0            & \ldots    & 0         & 1     & d^{(\ell)}_n \\   
		\end{pmatrix} 
		\text{, } 
		B_n(\bsd)=
		\begin{pmatrix}
		0      & -1         & 0         & \dots  & 0   \\   
		1     & d^{(\ell)}_1 & -1         & \ddots & \vdots\\      
		0      & 1        & d^{(\ell)}_2 & \ddots & 0      \\
		\vdots & \ddots    & \ddots    & \ddots & -1 \\
		0      & \ldots    & 0         & 1     & d^{(\ell)}_{n-1} \\   
		\end{pmatrix},
	\end{align*}
with $d^{(\ell)}_j:=d_{j \bmod \ell}$, and where we understand $d_0$ as $d_{\ell}$. It is then easily checked that 
\begin{equation*}
q_{n+1}(\delta) = \det A_n(\bsd)
\quad \text{ and } \quad
p_n(\delta) = \det B_n(\bsd),
\end{equation*}
where $p_n(\delta)/q_n(\delta)$ is the $n$th convergent of $\delta$ (see e.g.\ \cite[p.~10--11]{Perron}). In particular, for every $u \in \{ 0,1, \ldots , \ell-1\}$, we have
\begin{equation*}
c(\alpha_{\tau_u}) = \det A_{\ell} \left( \tau_u(\bsa)\right) + \det B_{\ell}\left( \tau_u(\bsa) \right)
\end{equation*}
and
\begin{equation*}
c(\alpha_{\sigma_u}) = \det A_{\ell} \left( \sigma_u(\bsa)\right) + \det B_{\ell}\left( \sigma_u(\bsa) \right),
\end{equation*}
where $\bsa=(a_1, \ldots , a_{\ell})$.

Let us first show that 
\begin{equation}
\label{eq:ctauisc}
c(\alpha_{\tau_u}) = c(\alpha) \quad \text{ for } \quad u = 0,1, \ldots , \ell-1.
\end{equation}
This is clearly the case when $u=0$, as $\tau_0$ is the identity operator on $\NN^{\ell}$ and $\alpha_{\tau_0}=\alpha$. For $u>0$, one can show that
\begin{equation*}
\det A_{\ell} \left( \tau_u (\bsa) \right) - \det B_{\ell} \left( \tau_{u+1} (\bsa) \right) = \det A_{\ell} \left( \tau_{u+1}(\bsa)\right) - \det B_{\ell} \left( \tau_u (\bsa) \right).
\end{equation*}
(This is attained by taking the Laplace expansion along appropriate rows and columns of the matrices above.) It follows that 
\begin{align*}
c(\alpha_{\tau_{u+1}}) &= \det A_{\ell} \left( \tau_{u+1}(\bsa)\right) + \det B_{\ell} \left( \tau_{u+1}(\bsa)\right) \\
&= \det A_{\ell} \left( \tau_{u}(\bsa)\right) + \det B_{\ell} \left( \tau_{u}(\bsa)\right) = c(\alpha_{\tau_u}),
\end{align*}
and thus \eqref{eq:ctauisc} holds.

Now let us verify that 
\begin{equation}
\label{eq:csigmaisc}
c(\alpha_{\sigma_u}) = c(\alpha) \quad \text{ for } \quad u = 0,1, \ldots , \ell-1.
\end{equation}
We note first that the operator $\sigma_u$ (for $u\neq 0$) can be expressed as a composition of $\tau_u$ and $\sigma_0$; namely 
$$\sigma_u (\bsa) = \sigma_0 (\tau_u(\bsa)) , \quad u \in \{ 1, 2, \ldots , \ell-1 \} .$$
Thus, if we can verify that 
\begin{equation}
\label{eq:sigmazero}
c(\alpha_{\sigma_0}) = c(\alpha),
\end{equation}
then the general case \eqref{eq:csigmaisc} will follow from \eqref{eq:ctauisc}. In fact, for $\sigma_0$ it is easily seen that
$$\det A_{\ell}(\sigma_0(\bsa)) = \det A_{\ell}(\tau_{\ell-1}(\bsa)),$$
and by Laplace expansion one can verify that also
$$\det B_{\ell}(\sigma_0(\bsa)) = \det B_{\ell}(\tau_{\ell-1}(\bsa)).$$
It follows that $c(\alpha_{\sigma_0})=c(\alpha_{\tau_{\ell-1}})$, which by \eqref{eq:ctauisc} implies \eqref{eq:sigmazero}. This completes the proof of Lemma \ref{BasicPropC}.
\end{proof}


\section{Properties of the sequence $(q_n)_{n\geq 0}$} \label{sec:propqn}
The main focus in this section is to attain a closed form for the sequence of best approximation denominators $(q_n)_{n\geq 0}$ for the irrational $\alpha=[0;\overline{a_1, \ldots, a_{\ell}}]$. We will see that having such a closed form enables us to formulate analogs of known properties for the Fibonacci sequence $(F_n)_{n\geq 0}=(0,1,1,2,3,5,8 \ldots )$, most notably of
\begin{align}
F_n\omega^n &= \frac{1}{\sqrt{5}} + \mathcal{O}(\omega^{2n})  \text{ for } n\geq 0; \label{eq:Fomega} \\
\frac{F_{n-1}}{F_n} &= \omega + \mathcal{O}(\omega^{2n}) \text{ for } n>0, \label{eq:Fquot}
\end{align}
where $\omega = (\sqrt{5}-1)/2$ is the fractional part of the golden mean. These two identities play a crucial role in the proof of Theorem \ref{thm:mvgolden} by Mestel and Verschueren. Likewise, the analogous identities for the sequence $(q_n)_{n\geq 0}$, formulated in Lemmas \ref{QnEstimates} and \ref{lem:qnquot} below, will be important for the proof of Theorem \ref{MainThm}.


\subsection{A connection to Lehmer sequences \label{sec:lehmer}}
We begin by establishing a closed form for the sequence $(q_n)_{n\geq 0}$ of best approximation denominators of $\alpha$. It turns out that this closed form can be expressed in terms of a \emph{Lehmer sequence}. Lehmer sequences were first introduced in \cite{Le30}, and are defined as follows. 

\begin{defi}\label{def:lehmer}
Let $R,Q\in \ZZ$ with $R>0$ and $R-4Q>0$. We define the Lehmer sequence $(L_n(R,Q))_{n\geq 0}$ with parameters $R$ and $Q$ by 
\begin{equation*}
\begin{aligned}
	L_{2n}(R,Q)&:=L_{2n-1}(R,Q)-QL_{2n-2}(R,Q)\\ 
	L_{2n+1}(R,Q)&:=RL_{2n}(R,Q)-QL_{2n-1}(R,Q)\\
	L_0(R,Q)&=0\\
	L_1(R,Q)&=1.
\end{aligned} 
\end{equation*}
\end{defi}
\noindent The closed form of the recurrence in Definition \ref{def:lehmer} is 
\begin{align}\label{closedFormLehmer}
	L_{n}(R,Q)=
	\begin{cases} 
		\dfrac{u^{n}-v^n}{u-v} & n \text{ odd,} \\[10pt]
		\dfrac{u^{n}-v^n}{u^2-v^2} & n \text{ even,} 
	\end{cases}
\end{align}
where $u$ and $v$ are the unique solutions of the equation $x^2 -\sqrt{R}x +Q=0$. 

We will consider only the Lehmer sequence with parameters $R=c(\alpha)^2$ and $Q=(-1)^{\ell-1}$. Accordingly, we write 
\begin{equation*}
L_n:= L_n\left(c(\alpha)^2, (-1)^{\ell-1}\right)
\end{equation*}
from now on. If we introduce the constants
\begin{align}
	& \label{eq:a} a= a(\alpha):=\frac{c(\alpha) + \sqrt{c(\alpha)^2 + 4(-1)^{\ell-1}}}{2};\\
	& \label{eq:b} b=b(\alpha):=\frac{c(\alpha) - \sqrt{c(\alpha)^2 + 4(-1)^{\ell-1}}}{2},
\end{align}
for the two distinct solutions of $x^2-c(\alpha)x+(-1)^{\ell}=0$, then it follows from \eqref{closedFormLehmer} that
\begin{align}\label{closedFormLn}
	L_{n}=
	\begin{cases} 
		\dfrac{a^n-b^n}{a-b} & n \text{ odd,} \\[10pt]
		\dfrac{a^n-b^n}{a^2-b^2} & n \text{ even.} 
	\end{cases}
\end{align}
By straightforward calculations one can verify that 
\begin{equation*}
a b = (-1)^{\ell} \quad \text{ and } \quad a+b = c(\alpha).
\end{equation*}
Moreover, we have $a>1$, and $b \in (-1,0)$ if $\ell$ is odd and $b \in (0,1)$ if $\ell$ is even. 
Finally, as a consequence of Lemma \ref{BasicPropC}, we have
\begin{align*}
a(\alpha) &= a(\alpha_{\tau_u}) = a(\alpha_{\sigma_u}); \\
b(\alpha) &= b(\alpha_{\tau_u}) = b(\alpha_{\sigma_u}),
\end{align*}
for every $u \in \{ 0,1, \ldots , \ell-1\}$, where we recall the definition of $\alpha_{\tau_u}$ and $\alpha_{\sigma_u}$ from \eqref{eq:deltatau} and \eqref{eq:deltasigma}.

Let us now formulate a closed form of the sequence of best approximation denominators $(q_n)_{n\geq 0}$ for $\alpha$. In fact, similar closed forms can be established for both $(q_n)_{n\geq 0}$ and $(p_n)_{n\geq 0}$. 
\begin{lem}\label{qnClosedForm}
For every $n=\ell m + k \geq 2\ell$, where $m\in \NN$ and $k \in \{ 0, 1, \ldots , \ell-1\}$, the approximation denominator $q_n$ for $\alpha=[0; \overline{a_1, \ldots , a_{\ell}}]$ is given by
	$$q_n = q_{\ell m +k} = \frac{1}{a-b} \left( (a^m-b^m) q_{\ell+k} + (-1)^{\ell-1} (a^{m-1}-b^{m-1})q_k \right) ,$$
where $a$ and $b$ are defined in \eqref{eq:a} and \eqref{eq:b}.
\end{lem}

\begin{lem}\label{pnClosedForm}
For every $n=\ell m + k \geq 2\ell$, where $m\in \NN$ and $k \in \{ 0, 1, \ldots , \ell-1\}$, the approximation numerator $p_n$ for $\alpha=[0; \overline{a_1, \ldots , a_{\ell}}]$ is given by
	$$p_n = p_{\ell m +k} = \frac{1}{a-b} \left( (a^m-b^m) p_{\ell+k} + (-1)^{\ell-1} (a^{m-1}-b^{m-1})p_k \right) ,$$
where $a$ and $b$ are defined in \eqref{eq:a} and \eqref{eq:b}. \end{lem}

It is a simple observation that Lemmas \ref{qnClosedForm} and \ref{pnClosedForm} may alternatively be formulated in terms of the Lehmer sequence \eqref{closedFormLn}. As we find this to be of independent interest, we formulate it as a theorem. 
\begin{thm}
\label{thm:lehmer}
For every $n=\ell m + k \geq 2\ell$, where $m\in \NN$ and $k \in \{ 0, 1, \ldots , \ell-1\}$, the convergents $p_n/q_n$ of $\alpha=[0; \overline{a_1, \ldots , a_{\ell}}]$ are given by
$$q_n = q_{\ell m +k}=\gamma_1^{(m)}L_m q_{\ell+k} + (-1)^{\ell-1}\gamma_2^{(m)} L_{m-1} q_k,$$ 
and 
$$p_n = p_{\ell m +k}=\gamma_1^{(m)}L_m p_{\ell+k} + (-1)^{\ell-1}\gamma_2^{(m)} L_{m-1} p_k,$$ 
where $a$ and $b$ are defined in \eqref{eq:a} and \eqref{eq:b}, $L_m$ is the Lehmer sequence \eqref{closedFormLn}, and 
\begin{equation*}
	\gamma_1^{(m)}:=\begin{cases} c(\alpha)& m \text{ even}\\ 1 & m \text { odd} \end{cases}  \quad \text{ and } \quad \gamma_2^{(m)}=\begin{cases} 1& m \text{ even}\\ c(\alpha) & m \text { odd} \end{cases}.
\end{equation*}
\end{thm}

As the proofs of Lemmas \ref{qnClosedForm} and \ref{pnClosedForm} are nearly identical, we include only the former.
\begin{proof}[Proof of Lemma \ref{qnClosedForm}]
By Lemma \ref{qnRecursion}, we have the recursion formula
$$q_n = c(\alpha)q_{n-\ell}+(-1)^{\ell-1}q_{n-2\ell} ,$$
whenever $n\geq 2\ell$. The corresponding polynomial characteristic equation is 
\begin{equation}\label{charctEqu}
	x^{2\ell}-c(\alpha)x^{\ell} + (-1)^{\ell}=0. 
\end{equation}
Substituting $y=x^{\ell}$, we get the equation $y^2 -c(\alpha)y+(-1)^{\ell}=0$, whose two solutions $a$ and $b$ are given in \eqref{eq:a} and \eqref{eq:b}, respectively. Now let $e_{\ell} = e^{2\pi i / \ell}$. The $2\ell$ unique solutions of \eqref{charctEqu} are 
\begin{equation}
x_v=a^{1/\ell} e_{\ell}^v \quad \text{ and } \quad x_{v+\ell} = b^{1/\ell} e_{\ell}^v 
\end{equation}
for $v=1,2,\ldots , \ell$. Accordingly, for an arbitrary $n \geq 2\ell$, the closed form of $q_n$ is 
\begin{equation}\label{eq:closedqn}
q_n= \sum_{v=1}^{2\ell} c_v x_v^n , 
\end{equation} 
where the constants $c_1, \ldots , c_{2\ell}$ are determined by the $2\ell$ first terms $q_0, \ldots , q_{2\ell -1}$. For a given $j\in \{ 1,2, \ldots , \ell \}$, inserting $q_{j-1}$ in \eqref{eq:closedqn} yields
\begin{equation*}
q_{j-1}=\sum_{v=1}^{2\ell} c_v x_v^{j-1} = a^{(j-1)/\ell} C_j^{(1)} + b^{(j-1)/\ell} C_j^{(2)} , 
\end{equation*}
where
\begin{equation*}
C_j^{(1)}=\sum_{v=1}^{\ell} c_v e_{\ell}^{v (j-1)} \quad \text{ and } \quad C_j^{(2)} = \sum_{v=\ell+1}^{2\ell} c_v e_{\ell}^{v (j-1)}.
\end{equation*}
Similarly, we have 
\begin{equation*}
q_{\ell+j-1}=  a^{1+(j-1)/\ell} C_j^{(1)} + b^{1+(j-1)/\ell} C_j^{(2)}.
\end{equation*}
Thus, the system of $2\ell$ equations determining the constants $c_{1}, \ldots , c_{2\ell}$ decouples into $\ell$ systems of $2$ equations in the variables $(C_j^{(1)}, C_j^{(2)})$, with $j=1,2,\ldots , \ell$. Solving these $\ell$ systems, we get
\begin{align*}
	C^{(1)}_j=\dfrac{1}{a^{(j-1)/\ell}}\dfrac{bq_{j-1}-q_{\ell+j-1}}{b-a} ; \\
	C^{(2)}_j=\dfrac{1}{b^{(j-1)/\ell}}\dfrac{q_{\ell+j-1}-aq_{j-1}}{b-a}.
\end{align*} 
Finally, for $n=\ell m + k \geq 2 \ell$, we thus have
\begin{align*}
	q_n &= q_{\ell m+k}  = \sum_{v=1}^{2\ell} c_v x_v^{\ell m +k} =a^{m+k/\ell}C_{k+1}^{(1)}+b^{m+k/\ell}C_{k+1}^{(2)} \\
	&=\frac{1}{a-b} \left( (a^m-b^m) q_{\ell+k} + (-1)^{\ell-1} (a^{m-1}-b^{m-1})q_k \right), 
\end{align*}
where in the final equality we have used that $ab=(-1)^{\ell}$. This completes the proof of Lemma \ref{qnClosedForm}.
\end{proof}


\subsection{Asymptotic behaviour of $q_n$ and $q_{n-1}/q_n$ \label{sec:consequences}}
In this section, we show that the closed form of $(q_n)_{n\geq 0}$ that was established in Lemma~\ref{qnClosedForm} can be used to formulate analogs of the Fibonacci identities \eqref{eq:Fomega} and \eqref{eq:Fquot} for the more general case of irrationals with a periodic continued fraction expansion. We begin with the simpler task of formulating an analog of \eqref{eq:Fomega}. Informally speaking, we will see that the constant $b$ in \eqref{eq:b} plays the role of the fractional part of the golden mean $\omega$.


\begin{lem}\label{QnEstimates} 
Let $\ell \in \NN$ and $k \in \{0,1, \ldots,\ell-1\}$ be fixed integers, and let $(p_n/q_n)_{n\geq 1}$ be the convergents for $\alpha=[0;\overline{a_1, \ldots , a_{\ell}}]$. For all integers $m\geq 2$, we have
\begin{equation}
\label{QnEstimatesFirst}
q_{\ell m+k}|b|^m = c_k + \mathcal{O}(|b|^{2m}),
\end{equation}
where 
\begin{equation}
\label{eq:ck}
c_k:=\frac{q_{\ell+k}-bq_k}{a-b},
\end{equation}
and $a$ and $b$ are given in \eqref{eq:a} and \eqref{eq:b}. Thus, we have $q_{\ell m +k}=\mathcal{O}(|b|^{-m})$.
\end{lem}
\begin{rem*}
Note that in the special case when $\alpha=\omega$ is the fractional part of the golden mean and $q_n=F_n$ is the Fibonacci sequence, we have $b=\omega$, and $c_k=c_0=1/\sqrt{5}$. Accordingly, Lemma \ref{QnEstimates} reduces to the Fibonacci identity \eqref{eq:Fomega} in this case.
\end{rem*}

\begin{proof}[Proof of Lemma \ref{QnEstimates}]
	Recall again the closed form 
	\begin{equation*}
		q_{\ell m+k} = \frac{1}{a-b}\left( (a^m-b^m)q_{\ell+k} +(-1)^{\ell-1}(a^{m-1}-b^{m-1})q_k \right) 
	\end{equation*}	
	from Lemma \ref{qnClosedForm}. Multiplying both sides by $|b|^m$ and using that $ab=(-1)^{\ell}$, we get
	\begin{equation*}
		q_{\ell m +k}|b|^m = \frac{1}{a-b}\left( q_{\ell+k} -bq_k + (-1)^{\ell m}b^{2m}(aq_k-q_{\ell+k})\right) = c_k + \mathcal{O}(|b|^{2m}),
	\end{equation*}
	with $c_k>0$ as in \eqref{eq:ck}.
\end{proof}


We now aim to establish an analog, or extension, of the Fibonacci identity \eqref{eq:Fquot}. This identity, which plays a crucial role in the work of Mestel and Verschueren \cite{MV15}, is a consequence of the fact that $F_{n-1}/F_n$ is the $n$th convergent of the golden mean $\omega$. Naturally, we cannot expect the same identity to hold for the general case of irrationals with a periodic continued fraction expansion. However, we will see in Lemma \ref{lem:qnquot} below that a similar identity can indeed be formulated. Needed for this lemma is the following estimation error of the $n$th convergent $p_n/q_n$ of $\alpha=[0;\overline{a_1, \ldots , a_{\ell}}]$ in terms of the constant $b$ in \eqref{eq:b}.

\begin{lem}\label{BasicPropQn}
	Let $n=\ell m +k\geq 2\ell$, where $m\in \NN$ and $k \in \{0,1, \ldots,\ell-1\}.$ Moreover, let $(p_n/q_n)_{n\geq 1}$ be the sequence of convergents for $\alpha=[0;\overline{a_1, \ldots , a_{\ell}}]$. We have that 
	\begin{equation}
\label{BasicPropQnFifth} 
q_n \alpha=q_{\ell m +k} \alpha = p_n + e_k b^m ,
\end{equation}
where
\begin{equation}
\label{eq:conste}
	e_{k} = \frac{(-1)^{k-1}}{q_{\ell}}\left| aq_{k}-q_{\ell+k} \right|, 
\end{equation} 
and $a$ and $b$ are given in \eqref{eq:a} and \eqref{eq:b}.
\end{lem}
\begin{rem*}
Note, in particular, that since $q_n= q_{\ell m+k} = \mathcal{O}(|b|^{-m})$ by Lemma~\ref{QnEstimates}, it follows from Lemma \ref{BasicPropQn} that 
\begin{equation}
\label{eq:estimate}
\alpha = \frac{p_n}{q_n}+ \mathcal{O}(|b|^{2m}). 
\end{equation}
\end{rem*}

\begin{proof}[Proof of Lemma \ref{BasicPropQn}]
As a preliminary step (note that this is not part of the statement), we show that \eqref{BasicPropQnFifth} holds when $m=1$ and $k=0$, that is 
\begin{equation}
\label{eq:firstcase}
q_{\ell}\alpha = p_{\ell}-b .
\end{equation}
It is well known that $\alpha$ is a root of the polynomial $q_{\ell}x^2+(q_{\ell+1}-p_{\ell})x-p_{\ell+1}$ (see e.g. \cite[p.~69]{Perron}), and since $\alpha>0$ we must have
\begin{align*}
\alpha&=\frac{-q_{\ell+1}+p_{\ell}+ \sqrt{(q_{\ell+1}-p_{\ell})^2+4p_{\ell+1}q_{\ell}}}{2q_{\ell}}\\
&=\frac{p_{\ell}}{q_{\ell}} + \frac{-(q_{\ell+1}+p_{\ell})+ \sqrt{(q_{\ell+1}+p_{\ell})^2-4q_{\ell+1}p_{\ell}+4p_{\ell+1}q_{\ell}}}{2q_{\ell}}
\end{align*}
Using that $c(\alpha)=q_{\ell+1}+p_{\ell}$ and $p_{\ell+1}q_{\ell}-q_{\ell+1}p_{\ell}=(-1)^{\ell-1}$, we thus get
$$q_{\ell} \alpha = p_{\ell} - \frac{c(\alpha)- \sqrt{c(\alpha)^2+ 4(-1)^{\ell-1}}}{2} = p_{\ell}-b.$$

Now let us see that \eqref{BasicPropQnFifth} holds for all $n = \ell m +k \geq 2\ell$.
We recall the closed forms 
	$$p_n = p_{\ell m +k} = \frac{1}{a-b} \left( (a^m-b^m) p_{\ell+k} + (-1)^{\ell-1} (a^{m-1}-b^{m-1})p_k \right) ,$$
and 
	$$q_n = q_{\ell m +k} = \frac{1}{a-b} \left( (a^m-b^m) q_{\ell+k} + (-1)^{\ell-1} (a^{m-1}-b^{m-1})q_k \right) ,$$
from Lemmas \ref{pnClosedForm} and \ref{qnClosedForm}, respectively. 
Multiplying the latter with $\alpha$, and using \eqref{eq:firstcase}, we get
	$$q_n\alpha = \frac{(p_{\ell}-b)}{q_{\ell}(a-b)}\left( (a^m-b^m)q_{\ell+k}+(-1)^{\ell-1}(a^{m-1}-b^{m-1})q_k \right).$$
For ease of notation, let us write $A=(a^m-b^m)$ and $B=(-1)^{\ell-1}(a^{m-1}-b^{m-1})$. We then have
\begin{align*}
	q_n\alpha - p_n &= \frac{1}{q_{\ell}(a-b)} \left( (p_{\ell}-b) \left( Aq_{\ell+k}+Bq_k\right)-(Aq_{\ell}p_{\ell+k}+Bq_{\ell}p_k) \right) \\
	&= \frac{1}{q_{\ell}(a-b)} \left( A(q_{\ell+k}p_{\ell}-bq_{\ell+k}-q_{\ell}p_{\ell+k})+B(p_{\ell}q_k-bq_k-q_{\ell}p_k)\right).
\end{align*}
Now recall from Lemma \ref{BasicPropQP} that 
$$q_{\ell+k}p_{\ell}-q_{\ell}p_{\ell+k}=(-1)^{\ell}q_k$$
(part \ref{BasicPropQPFirst} with $m=\ell$ and $n=k$) and  
$$p_kq_{\ell}= q_{\ell+k}-q_{\ell+1}q_k $$
(part \ref{BasicPropQPSecond} with $m=\ell$ and $n=k$). Inserting this above, we get
\begin{align*}
q_n\alpha - p_n &= \frac{1}{q_\ell(a-b)} \left( A((-1)^{\ell}q_k-bq_{\ell+k}) + B(-q_{\ell+k}+q_k(q_{\ell+1}+p_{\ell}-b)) \right) \\
&= \frac{1}{q_{\ell}} \left( -\left( \frac{Ab+B}{a-b}\right)q_{\ell+k}+ \left( \frac{(-1)^{\ell}A+Ba}{a-b}\right)q_k\right),
\end{align*}
where we have used that $q_{\ell+1}+p_{\ell}-b= c(\alpha)-b=a$. From $ab=(-1)^{\ell}$, it follows that
\begin{equation*}
\frac{Ab+B}{a-b} = b^m \quad \text{ and } \quad \frac{(-1)^{\ell}A+Ba}{a-b} = ab^m .
\end{equation*}
Accordingly, we have 
$$q_n\alpha - p_n = \frac{b^m}{q_{\ell}}(aq_k-q_{\ell+k}).$$
Finally, we know from the general theory of continued fractions that $p_n/q_n$ is greater than $\alpha$ if $n$ is even and smaller than $\alpha$ if $n$ is odd. We thus get
\begin{align*}
q_n\alpha-p_n &= (-1)^{n-1} \left| \frac{b^m}{q_{\ell}}(aq_k-q_{\ell+k}) \right| \\
&=(-1)^{\ell m+k-1}\left((-1)^{\ell}b\right)^m\cdot \frac{1}{q_{\ell}}\left| aq_k-q_{\ell+k} \right| = e_kb^m,
\end{align*}
with $e_k$ given in \eqref{eq:conste}. This completes the proof of Lemma \ref{BasicPropQn}.
\end{proof}


With Lemma \ref{BasicPropQn} established, we may now formulate the following analog of the Fibonacci sequence \eqref{eq:Fquot} for the general case of irrationals with a periodic continued fraction expansion.
\begin{lem}
\label{lem:qnquot}
Let $\ell \in \NN$ and $k \in \{ 0,1, \ldots , \ell -1 \}$ be fixed integers, and let $(p_n/q_n)_{n\geq 1}$ be the convergents for $\alpha=[0;\overline{a_1, \ldots , a_{\ell}}]$. We have that
\begin{equation}
\label{eq:qnquot}
\frac{q_{\ell m +k-1}}{q_{\ell m +k }} = \frac{p_{\ell m +k}(\alpha_{\sigma_k})}{q_{\ell m + k}(\alpha_{\sigma_k})} = \alpha_{\sigma_k} + \mathcal{O}(|b|^{2m}),
\end{equation}
where $b$ is given in \eqref{eq:b}.
\end{lem}
\begin{rem*}
In the special case when $\alpha=\omega$ is the fractional part of the golden mean and $q_n=F_n$ is the Fibonacci sequence, we have $k=0$, $\alpha_{\sigma_0}= \alpha=\omega$ and $b=\omega$. Accordingly, Lemma \ref{lem:qnquot} reduces to the Fibonacci identity \eqref{eq:Fquot} in this case.
\end{rem*}
\begin{proof}[Proof of Lemma \ref{lem:qnquot}]
For ease of notation, we write $n= \ell m +k$. Let us first see that $q_{n-1}/q_n=p_n(\alpha_{\sigma_k})/q_n(\alpha_{\sigma_k})$. We treat only the case $k\geq 1$ (the case $k=0$ is similar). On the one hand, we have 
\begin{equation*}
\frac{p_{n}(\alpha_{\sigma_k})}{q_n(\alpha_{\sigma_k})}=[0;\underbrace{a_{k-1},\ldots,
		a_1,a_{\ell},\ldots,a_k, \ldots, a_{k-1},\ldots,a_1,a_{\ell},\ldots,a_k}_{m \text{ times}}, a_{k-1},\ldots,a_1].
\end{equation*}
On the other hand, using the recursion formula for $q_n$, we get
\begin{align*}
	\cfrac{q_{n-1}}{q_n} &= \cfrac{q_{\ell m + k-1}(\alpha)}{q_{\ell m+k}(\alpha)} = \cfrac{1}{a_{k-1} + \cfrac{q_{\ell m+k-2}}
			{q_{\ell m+k-1}}}	=\cfrac{1}{a_{k-1} + \cfrac{1}{a_{k-2}+\cfrac{q_{\ell m+k-3}}{q_{\ell m+k-2}}}} \\
		&=\ldots	= [0;a_{k-1},\ldots,a_1,\underbrace{ a_{\ell},\ldots,a_1,\ldots,a_{\ell},\ldots,a_1}_{m 
			\text{ times}}].
\end{align*}
Thus, these quotients are equal. Finally, it follows from \eqref{eq:estimate} and the fact that $b(\alpha_{\sigma_k})=b(\alpha)$ for every $k\in \{ 0,1, \ldots , \ell-1\}$ that
\begin{equation*}
\frac{p_{n}(\alpha_{\sigma_k})}{q_n(\alpha_{\sigma_k})} = \alpha_{\sigma_k} + \mathcal{O}(|b(\alpha_{\sigma_k})|^{2m}) = \alpha_{\sigma_k} + \mathcal{O}(|b|^{2m}) .
\end{equation*}
\end{proof}


We conclude this section by a more thorough investigation of the constants $c_k$ and $e_k$ in \eqref{eq:ck} and \eqref{eq:conste}. More specifically, we consider the absolute value of their product $|c_ke_k|$. This quantity will repeatedly appear in the proof of Theorem \ref{MainThm}, and the following lemma on $|c_ke_k|$ will then be useful. 
\begin{lem}
\label{lem:absckek}
For $c_k$ and $e_k$ given in \eqref{eq:ck} and \eqref{eq:conste}, we have that 
\begin{equation}
\label{eq:altckek}
|c_ke_k| = \frac{q_{\ell}(\alpha_{\tau_k})}{a-b} = \frac{q_{\ell}(\alpha_{\tau_k})}{c(\alpha_{\tau_k})-2b} \, ,
\end{equation}
with $a$ and $b$ given in \eqref{eq:a} and \eqref{eq:b}. It follows that $|c_ke_k|<1$ for each $k=0,1, \ldots , \ell-1$. 
\end{lem}
\begin{proof}
Recalling the definition of $c_k$ and $e_k$, we have
$$|c_ke_k| = \frac{|(q_{\ell+k}-bq_k)(aq_k-q_{\ell + k})|}{q_{\ell}(a-b)} = \frac{|q_k\left( c(\alpha)q_{\ell+k}+(-1)^{\ell-1}q_k \right) -q_{\ell+k}^2|}{q_{\ell}(a-b)},$$
where we have used that $a+b=c(\alpha)$ and $ab=(-1)^{\ell}$. We now look at the numerator in this expression. Using the recursion formula in Lemma \ref{qnRecursion}, and Lemma \ref{BasicPropQP}\ref{BasicPropQPSecond} (with $r=\ell+k$), we have
$$c(\alpha)q_{\ell+k}+(-1)^{\ell-1}q_k = q_{2\ell+k} = q_{\ell+1}q_{\ell+k}+q_{\ell}p_{\ell+k} .$$
Inserting this in the numerator, we get
$$|c_ke_k| = \frac{\left| q_{\ell}q_kp_{\ell + k} + q_{\ell+k}(q_kq_{\ell+1}-q_{\ell+k})\right|}{q_{\ell}(a-b)} = \frac{\left|q_k p_{\ell+k} - q_{\ell+k}p_k \right|}{(a-b)},$$
where for the last equality we have used that $q_kq_{\ell+1}-q_{\ell+k}= -q_{\ell}p_k$ (Lemma~\ref{BasicPropQP}\ref{BasicPropQPSecond} with $r=k$). It now follows from Lemma \ref{BasicPropQP}\ref{BasicPropQPFirst} (with $m=k$ and $n=\ell$) and $c(\alpha)=c(\alpha_{\tau_k})$ that 
\begin{equation*}
|c_ke_k|= \frac{q_{\ell}(\alpha_{\tau_k})}{a-b} = \frac{q_{\ell}(\alpha_{\tau_k})}{c(\alpha)-2b} =  \frac{q_{\ell}(\alpha_{\tau_k})}{c(\alpha_{\tau_k})-2b}, 
\end{equation*}
which confirms \eqref{eq:altckek}. 

From \eqref{eq:altckek} it easily follows that $|c_ke_k|<1$. For $\ell=1$, we get 
$$|c_0e_0| = \frac{1}{\sqrt{a_1^2+4}}<1.$$ 
For $\ell\geq 2$, we have $p_{\ell}(\alpha_{\tau_k})\geq 1$, and accordingly
\begin{equation*}
|c_ke_k| = \frac{q_{\ell}(\alpha_{\tau_k})}{q_{\ell+1}(\alpha_{\tau_k})+p_{\ell}(\alpha_{\tau_k})-2b} < \frac{q_{\ell}(\alpha_{\tau_k})}{q_{\ell+1}(\alpha_{\tau_k})-1} \leq 1
\end{equation*}
for each $k=0,1, \ldots , \ell -1$. This completes the proof of Lemma \ref{lem:absckek}.
\end{proof}


\section{Decomposing $Q_{\ell m +k}(\alpha)$ \label{sec:decomp}}
The aim of this section is to decompose the product of sines $Q_{\ell m +k}(\alpha)$ in \eqref{eq:defqn} into three subproducts 
\begin{align*}
Q_{\ell m +k}(\alpha) &= A_m B_m C_m\\
&= \left| 2q_n \sin \pi e_kb^m \right| \cdot \left| \prod_{t=1}^{q_n-1} \frac{s_{mt}}{2\sin(\pi t/q_n)} \right| \cdot \prod_{t=1}^{q_n-1}\left( 1- \frac{s_{m0}^2}{s_{mt}^2}\right)^{1/2},
\end{align*}
where $n=\ell m+k$ and $s_{mt}$ is a perturbed rational sine function defined in \eqref{eq:sm} below. This decomposition is achieved by substituting the identity $\alpha = p_{n}/q_{n} + e_kb^m/q_{n}$ from Lemma \ref{BasicPropQn} into the definition of $Q_{\ell m + k}(\alpha)$, which in turn allows us to view $Q_{\ell m +k}(\alpha)$ as a perturbation of the rational sine product $\prod_{r=1}^{q_n-1} |2 \sin \pi r (p_{n}/q_{n})|$. Accordingly, proving Theorem \ref{MainThm} is a matter of showing that these perturbations have a suitable behaviour. For the latter task, it is a disadvantage that the perturbations $re_kb^m/q_{n}$ do not sum up to zero. However, by a rebasing of the argument one can attain a set of shifted perturbations $e_kb^m(r/q_n-1/2)$ which \emph{do} sum up to zero, and this approach eventually leads to the decomposition above. 

At the end of this section, we show the straightforward convergence of $A_m$ as $m\rightarrow \infty$. The convergence of $B_m$ and $C_m$ requires more work, and is treated in subsequent Sections \ref{sec:Cm} and \ref{sec:Bm}. Finally, we conclude that $Q_{\ell m+k}(\alpha)$ must be convergent for every fixed $\ell \in \NN$ and $k \in \{0,1, \ldots , \ell-1\}$ in Section~\ref{sec:proof}.


\subsection{Important families of sequences}
Before we decompose $Q_{\ell m +k}(\alpha)$ into subproducts $A_m$, $B_m$ and $C_m$ in Section \ref{subsec:decomp}, let us introduce certain families of sequences which enter into the decomposition. For integers $m\geq 1$ and $t \in \{ 0,1, \ldots , q_{\ell m +k}-1\}$, we define:
\begin{align}
	&s_{mt}:=2 \sin \pi\left(\frac{t}{q_{\ell m +k}} - |e_{k}b^m| \left(\left\{ \frac{tq_{\ell m +k-1}} {q_{\ell m +k}}\right\}-\frac{1}{2}\right)\right) \label{eq:sm}\\
	&\xi_{mt}:=\left\{t\frac{q_{\ell m +k -1}}{q_{\ell m +k}}\right\}-\frac{1}{2} \label{eq:xim}\\
	&\xi_{\infty t}:=\left \{t \alpha_{\sigma_k} \right\} - \frac{1}{2} \label{eq:xinfty}\\
	&h_{mt}:=\cot\left(\frac{\pi t}{q_{\ell m +k}}\right)\sin \pi |e_{k}b^m|\xi_{mt} \label{eq:hm}\\
	&h_{\infty t}:=\frac{|c_ke_k| \xi_{\infty t}}{t},~(t\neq 0) \label{eq:hinfty}
\end{align}
Combining \eqref{eq:sm} and \eqref{eq:xim}, we get
\begin{equation}
\label{eq:smalt}
s_{mt} = 2 \sin \pi \left( \frac{t}{q_{\ell m +k}} - \xi_{mt} |e_kb^m |\right) .
\end{equation}
It is clear from the definition of $\xi_{mt}$ that $|\xi_{mt}|\leq1/2$, and since $|b|<1$, we recognize $s_{mt}$ as the perturbation of a rational sine, where the perturbation tends to zero as $m$ increases. As we have already seen, the sequence $s_{mt}$ plays a crucial role in the decomposition of $Q_{\ell m +k}(\alpha)$.

The sequences $h_{mt}$ and $h_{\infty t}$ will not enter the story until the convergence of the subproduct $B_m$ is considered in Section \ref{sec:Bm}. Nevertheless, we introduce them at this early stage. 
\begin{lem}
\label{lem:propsm}
Let $s_{mt}$, $\xi_{mt}$, $\xi_{\infty t}$, $h_{mt}$ and $h_{\infty t}$ be the sequences given in \eqref{eq:sm} --\eqref{eq:hinfty}. We have that:
	\begin{enumerate}[label=(\alph*)]
		\item $s_{mt} = s_{m(q_{\ell m +k}-t)}$, $h_{mt} = h_{m(q_{\ell m +k}-t)}$ and $\xi_{mt}=-\xi_{m(q_{\ell m+k}-t)}$ for all $t \in \{ 0,1, \ldots , q_{\ell m +k} -1 \}$.\label{it:propsm1}
		\item $s_{mt} > s_{m0}$ for all $t \in \{ 1,2, \ldots , q_{\ell m +k} -1 \}$ if $m$ is sufficiently large.\label{it:propsm2}
		\item $\xi_{mt}-\xi_{\infty t} = \mathcal{O}(tb^{2m})$, and thus for any fixed $t \in \NN$, we have $$\lim_{m \rightarrow \infty} \xi_{mt} = \xi_{\infty t} .$$			\label{it:propsm3}
		\item $h_{mt} - h_{\infty t} = \mathcal{O}(tb^{2m})$, and thus for any fixed $t \in \NN$, we have $$\lim_{m \rightarrow \infty} h_{mt} = h_{\infty t} .$$ 				\label{it:propsm4}
	\end{enumerate}
\end{lem}
\begin{proof}
For ease of notation, let us again write $n=\ell m+k$. 

We first verify \ref{it:propsm1}. The fact that $\{-x\}= 1-\{ x\}$ for $x\in \mbb{R}\setminus \mbb{Z}$ immediately implies $\xi_{mt}=-\xi_{m(q_n-t)}$. Combining this with $\sin(\pi - x) = \sin x$, we get
\begin{align*}
	s_{m(q_n-t)} &= 2 \sin \pi\left(\frac{q_n-t}{q_n} - \xi_{m(q_n-t)}|e_{k}b^m| \right) \\
	&= 2 \sin \pi\left( 1- \left( \frac{t}{q_n}- \xi_{mt} |e_{k}b^m| \right) \right) = s_{mt},
\end{align*}
and likewise since $\cot(\pi-x) =-\cot x$ and $\sin x$ is an odd function, we have
\begin{align*}
h_{m(q_n-t)} = \cot \pi \left( 1- \frac{t}{q_n}\right) \sin \left( - \pi |e_kb^m| \xi_{mt}\right) = h_{mt}
\end{align*}
for every $t \in \{ 0,1, \ldots , q_n-1\}$.

Now let us verify \ref{it:propsm2}. In light of \ref{it:propsm1}, it is enough to verify $s_{mt} > s_{m0}$ for $t \in \{ 1,2, \ldots ,  \lfloor q_n/2 \rfloor \}$.
Writing $s_{mt}$ as in \eqref{eq:smalt}, and recalling that $|\xi_{mt}| < 1/2$ for these values of $t$, it is clear that $s_{mt} > s_{m(t-1)}$, and in particular
\begin{equation*}
s_{m1} > 2 \sin \pi \left( \frac{1}{q_n} -\frac{1}{2}|e_kb^m| \right) > 2\sin \frac{\pi |e_kb^m|}{2} = s_{m0} ,
\end{equation*}
if only $|e_kb^m|<1/q_n$. This in turn follows from Lemmas \ref{QnEstimates} and \ref{lem:absckek}, as
\begin{equation*}
q_n |e_kb^m| = |c_ke_k| + \mathcal{O}(|b|^{2m}) < 1
\end{equation*}
for sufficiently large values of $m$. 

Finally, we verify \ref{it:propsm3} and \ref{it:propsm4}. It follows directly from Lemma \ref{lem:qnquot} that 
\begin{equation*}
\xi_{mt} = \{ t\alpha_{\sigma_k}\} - \frac{1}{2} + \mathcal{O}(tb^{2m}) = \xi_{\infty t} + \mathcal{O}(tb^{2m}),
\end{equation*}
which confirms \ref{it:propsm3}. For property \ref{it:propsm4}, we use $\cot x = (1/x)(1+\mathcal{O}(x^2))$ and $\sin x = x(1+\mathcal{O}(x^2))$ to rewrite $h_{mt}$ as
\begin{equation*}
h_{mt} = \frac{q_n|e_kb^m|\xi_{mt}}{t} \left( 1+\mathcal{O}(t^2b^{2m}) \right),
\end{equation*}
where we have also exploited that $1/q_n=\mathcal{O}(|b|^{m})$. Moreover, since $q_n|b|^{m}=c_k+\mathcal{O}(b^{2m})$ by Lemma \ref{QnEstimates}, we get
\begin{equation*}
h_{mt} =  \frac{|c_ke_k| \xi_{mt}}{t} \left( 1+ \mathcal{O}(t^2b^{2m}) \right) = \frac{|c_ke_k|\xi_{mt}}{t} + \mathcal{O}(tb^{2m}),
\end{equation*}
and finally by recalling property \ref{it:propsm3} it follows that $h_{mt} = h_{\infty t} + \mathcal{O}(tb^{2m})$. This confirms \ref{it:propsm4}, and completes the proof of Lemma \ref{lem:propsm}.
\end{proof}


\subsection{Decomposition of $Q_{\ell m +k}(\alpha)$ \label{subsec:decomp}}
We are now equipped to decompose the sine product $Q_{\ell m+k}(\alpha)$.

\begin{lem}
\label{lem:decomp}
Fix $k \in \{ 0,1, \ldots , \ell-1\}$, and for integers $m\geq 1$ and $t \in \{ 0,1, \ldots , q_{\ell m+k}-1\}$, let $s_{mt}$ be given in \eqref{eq:sm}.
The product of sines $Q_{\ell m +k}(\alpha)$ can be written as
$$Q_{\ell m + k}(\alpha) = \prod_{r=1}^{q_{\ell m +k}} |2\sin \pi r \alpha | = A_mB_mC_m ,$$ 
where:
\begin{align}
A_m &= \left| 2q_{\ell m +k} \sin \pi e_k b^m \right| \label{eq:Am}\\
B_m &= \left| \prod_{t=1}^{q_{\ell m +k}-1} \frac{s_{mt}}{2 \sin (\pi t / q_{\ell m +k})} \right| \label{eq:Bm} \\
C_m &= \prod_{t=1}^{(q_{\ell m +k}-1)/2} \left( 1- \frac{s_{m0}^2}{s_{mt}^2} \label{eq:Cm} \right)
\end{align}
\end{lem}

\begin{proof}
Again we introduce $n=\ell m +k$ for ease of notation. We then have $Q_{\ell m +k}(\alpha) = Q_n(\alpha) = \prod_1^{q_n} |2\sin \pi r \alpha|$, and 
\begin{align*}
	Q^2_{n}(\alpha) &=\left(2\sin \pi q_n \alpha \right)^2\left(\prod_{r=1}^
		{q_n-1} 2 \sin \pi r \alpha \right)^2\\
	&=\left(2\sin \pi q_n \alpha \right)^2\prod_{r=1}^{q_n-1} \left( 2 \sin \pi r \alpha  \right) \left(2 \sin \pi 
		(q_n-r) \alpha \right) \\
	&=\left(2\sin \pi q_n \alpha \right)^2\prod_{r=1}^{q_n-1} 2 \left( \cos(2\pi r \alpha-\pi 
		q_n \alpha) - \cos \pi q_n \alpha \right).\\
\end{align*}
For the last equality we have used the identity $\sin(x)\sin(y)= (\cos(x-y)-\cos(x+y))/2$. Inserting $q_n \alpha = p_n+e_kb^m$ from Lemma \ref{BasicPropQn} in the expression above, we get
\begin{align*} 
	Q_n^2(\alpha)&=\left(2\sin \pi e_kb^m \right)^2\prod_{r=1}^{q_n-1} 2(-1)^{p_n} \left( \cos \left(2\pi 
		r \alpha - \pi e_kb^m \right) - \cos \pi e_kb^m \right)\\
	&=\left(2\sin \pi e_kb^m \right)^2 \prod_{r=1}^{q_n-1}4 \left(\sin^2\left(\pi 
		r \alpha- \frac{\pi}{2} e_{k}b^m\right)-\sin^2\left(\frac{\pi}{2} e_{k}b^m\right)\right).\\
\end{align*}
Observe that we have used the identity $\cos(x)=1-2\sin^2(x/2)$ and that $(-1)^{(p_n+1)(q_n-1)} =1$. The latter follows from the fact that $\gcd(p_n,q_n)=1$, and accordingly either $(p_n+1)$ or $(q_n-1)$ is an even number. This concludes the rebasing of the argument described in the introduction to this section.

We now aim to express $Q_n^2(\alpha)$ as a product of perturbed rational sines. Again we use the identity $\alpha = p_n/q_n + e_kb^m/q_n$ from Lemma \ref{BasicPropQn} to get
\begin{align*}
	\sin^2 \left(\pi r \alpha- \frac{\pi}{2} e_{k}b^m \right) = \sin^2 \pi\left(\frac
		{rp_{n}}{q_n}+e_{k}b^m \left(\frac{r}{q_n}-\frac{1}{2}\right)\right).
\end{align*}
By the substitution $t= rp_n \bmod q_n$, and recalling from \eqref{eq:pq} that $p_nq_{n-1} = (-1)^n \bmod q_n$, we have
\begin{align*}
	\sin^2 \left(\pi r \alpha- \frac{\pi}{2} e_{k}b^m \right) &= \sin^2 \pi\left(\frac{rp_{n} \bmod{q_n}}{q_n}+e_{k}b^m\left(\frac{r}{q_n}-\frac{1}
	{2}\right)\right) \\
	&= \sin^2 \pi\left(\frac{t}{q_n}+ e_{k}b^m \left(\frac{(-1)^ntq_{n-1} \bmod{q_n}}{q_n}- \frac{1}{2}\right)\right) \\
	&=\frac{1}{4}s_{mt}^2,
\end{align*}
with $s_{mt}$ given in \eqref{eq:sm}, and where we have used $e_kb^m=(-1)^{n-1}|e_kb^m|$ and
$$\frac{(-1)^ntq_{n-1} \bmod{q_n}}{q_n}-\frac{1}{2}= \left\{\frac{(-1)^ntq_{n-1}}{q_n} \right\} - \frac{1}{2} = (-1)^n\left(\left\{\frac{tq_{n-1}}{q_n}\right\} - \frac{1}{2} \right).$$ 
As $r$ runs through the values $1,2,\ldots , q_n-1$, so does $t=r p_n \bmod q_n$. Accordingly, we get
\begin{align*}
	Q^2_n(\alpha) &=(2\sin \pi e_kb^m)^2 \prod_{t=1}^{q_n-1}
		 \left( s^2_{mt}-s^2_{m0} \right)\\
	&=(2\sin\pi e_{k}b^m)^2 \prod_{t=1}^{q_{n}-1}s^2_{mt}\prod_{t=1}^{q_{n}-1} \left(1-
		\frac{s^2_{m0}}{s^2_{mt}} \right)\\
	&=(2q_n\sin \pi e_{k}b^m)^2 \prod_{t=1}^{q_{n}-1}\frac{s^2_{mt}}{4\sin^2(\pi t/ q_n)}\prod_{t=1}^{q_{n}-1} \left(1-\frac{s^2_{m0}}{s^2_{mt}}\right).
\end{align*}
For the last equality above we have used the well-known identity
\begin{equation*}
\prod_{r=1}^{q-1} 2 \sin \left( \frac{\pi rp}{q} \right) = q 
\end{equation*}
whenever $p,q \in \ZZ$ satisfy $\gcd (p,q) =1$ (see e.g.\ \cite{mullin} for a nice proof). Finally, we recall from Lemma \ref{lem:propsm}\ref{it:propsm1} that $s_{mt}=s_{m(q_n-t)}$ and hence $s_{mt}^2 = s_{m(q_n-t)}^2$ for every $t \in \{ 0, 1, \ldots , q_n-1\}$.
With our generalized notion of products introduced in Section \ref{sec:notation}, we thus get
\begin{align*}
	\prod_{t=1}^{q_{n}-1}\left( 1-\frac{s^2_{m0}}{s^2_{mt}} \right)= \prod_{t=1}^{(q_n-1)/2} \left(1-\frac{s^2_{m0}}{s^2_{mt}}\right)^2
\end{align*}
Inserting this in the expression for $Q_n^2(\alpha)$ above and taking the square root of both sides, we arrive at
$$Q_n(\alpha) = Q_{\ell m +k}(\alpha) = A_mB_mC_m ,$$
where $A_m$, $B_m$ and $C_m$ are given in \eqref{eq:Am}, \eqref{eq:Bm} and \eqref{eq:Cm}, respectively.
\end{proof}


\subsection{Convergence of $A_m$ \label{sec:Amconv}}
Let us now see that $A_m$ in \eqref{eq:Am} converges as $m\rightarrow \infty$. Since $\sin x = x+\mathcal{O} (x^3)$, we have
\begin{equation*}
A_m  = \left| 2q_{\ell m +k} \left( \pi e_k b^m + \mathcal{O}(b^{3m}) \right) \right| .
\end{equation*}
By Lemma \ref{QnEstimates} and $|b|<1$ it thus follows that
\begin{equation}
\label{eq:limAm}
\lim_{m \rightarrow \infty} A_m = 2\pi |e_kc_k|,
\end{equation}
where $c_k$ and $e_k$ are the constants given in \eqref{eq:ck} and \eqref{eq:conste}, respectively. Alternatively, using the expression for $|e_kc_k|$ given in Lemma \ref{lem:absckek}, we have
\begin{equation*}
\lim_{m \rightarrow \infty} A_m = \frac{2\pi q_{\ell}(\alpha_{\tau_k})}{(a-b)}.
\end{equation*}


\section{Convergence of $C_m$ \label{sec:Cm}}
In this section we show that the product
$$C_m = \prod_{t=1}^{(q_{\ell m +k}-1)/2} \left( 1- \frac{s_{m0}^2}{s_{mt}^2} \right)$$ 
is convergent. This is not quite straightforward, as it is not obvious that the sequence $(C_m)_{m\geq1}$ is monotonically decreasing. However, we will see that $(C_m)_{m\geq 1}$ is comparable to a monotonically decreasing sequence of products bounded below by a positive number.
\begin{thm}
\label{thm:convCm}
The sequence $C_m$ converges to the strictly positive limit
\begin{equation}
\label{eq:convCm}
\lim_{m\rightarrow \infty} C_m = \prod_{t=1}^{\infty} \left( 1- \frac{1}{4\left( t/ |c_ke_k|-\xi_{\infty t} \right)^2}\right),
\end{equation}
where $|c_ke_k|$ is given in \eqref{eq:altckek} and $\xi_{\infty t} = \{ t \alpha_{\sigma_k}\} -1/2$. 
\end{thm}

We will need the following Lemma for proving Theorem \ref{thm:convCm}.
\begin{lem}[{\cite[Lemma 4.3]{MV15}}]\label{ProductSumLemma}
	For $n\geq 2$ and real numbers $a_t$, $t =1,2, \ldots , n$, satisfying $A:=\sum_{t=1}^{n}|a_t|<1$, we have
	$$1- A<\prod_{t=1}^n(1+a_t) <\frac{1}{1-A}.$$
\end{lem}
\begin{rem*}
In fact, what we will need is that
$$1-A < \prod_{t=1}^n (1-|a_t|) < \frac{1}{1-A}.$$
This is an immediate consequence of Lemma \ref{ProductSumLemma}.
\end{rem*}

\begin{proof}[Proof of Theorem \ref{thm:convCm}]
For ease of notation we again write $n=\ell m + k$, and begin by developing estimates for the quotients $s_{m0}/s_{mt}$. We have 
\begin{equation*}
s_{m0} = 2 \sin \left( \pi |e_kb^m|/2 \right) = \pi |e_kb^m| \left( 1+\mathcal{O}(b^{2m})\right),
\end{equation*}
and for $t\geq 1$ it follows from Lemmas \ref{QnEstimates} and \ref{lem:propsm}\ref{it:propsm3} that
\begin{equation}
\label{eq:smapprox}
\begin{aligned}
s_{mt} &= 2 \sin \pi \left( \frac{t}{q_n} - |e_kb^m|\xi_{mt} \right) \\
&= 2 \sin \pi t |b|^m \left( \frac{1}{c_k} - \frac{|e_k|\xi_{\infty t}}{t} + \mathcal{O}(b^{2m}) \right).
\end{aligned}
\end{equation}
We now split the values of $t$ at $\eta = \lceil |b|^{-3m/5}\rceil$, and treat $t \leq \eta$ and $t > \eta$ separately in order to find appropriate bounds on $s_{mt}$ in \eqref{eq:smapprox}. For $t> \eta$, we use $\sin x \geq 2x/\pi$ for $x \in [0, \pi/2]$ to obtain
\begin{equation*}
s_{mt} \geq 4t|b|^m \left( \frac{1}{c_k} - \frac{|e_k|\xi_{\infty t}}{t} + \mathcal{O}(b^{2m}) \right).
\end{equation*}
Recall that $c_k>0$ and $|\xi_{\infty t}| \leq 1/2$. Thus, for sufficiently large $m$ (and thereby sufficiently large $t$), we have $s_{mt} > 2\eta |b|^m/c_k$ and 
\begin{equation*}
\frac{s_{m0}}{s_{mt}} \leq \frac{\pi |e_kb^m|\left( 1+\mathcal{O}(b^{2m}) \right)}{2\eta|b|^m/c_k} = \frac{\pi |c_ke_k|}{2\eta} \left( 1+\mathcal{O}(b^{2m}) \right) = \mathcal{O}(\eta^{-1}).
\end{equation*}
It follows that 
\begin{equation*}
\sum_{t=\eta+1}^{(q_n-1)/2} \frac{s_{m0}^2}{s_{mt}^2} \leq q_n \cdot \mathcal{O}(\eta^{-2}) = \mathcal{O}(|b|^{m/5}),
\end{equation*}
and accordingly this sum is convergent and smaller than one for sufficiently large $m$. Thus, by Lemma \ref{ProductSumLemma} we get
\begin{equation}
\label{eq:largeetabound}
1 \geq \prod_{t=\eta+1}^{(q_n-1)/2} \left( 1-\frac{s_{m0}^2}{s_{mt}^2} \right) >1- \sum_{t=\eta+1}^{(q_n-1)/2} \frac{s_{m0}^2}{s_{mt}^2} \geq 1-\mathcal{O}(|b|^{m/5}).
\end{equation}

Now consider $t \leq \eta$. It is clear from \eqref{eq:smapprox} that by choosing $m$ sufficiently large, the argument in the sine function $s_{mt}$ can be made arbitrarily small in this case. Applying $\sin x = x+\mathcal{O}(x^3)$, we get
\begin{align*}
s_{mt} &= 2\pi |b|^m \left( \frac{t}{c_k}- |e_k| \xi_{\infty t} + \mathcal{O}(tb^{2m}) \right) + \mathcal{O}(|b|^{6m/5}) \\
&= \pi |e_kb^m| \left( u_t + \mathcal{O}(|b|^{m/5})\right),
\end{align*}
where we have introduced the notation
\begin{equation}
\label{eq:ut}
u_t = 2 \left( \frac{t}{|c_ke_k|} - \xi_{\infty t} \right) = 2 \left( \frac{t}{|c_ke_k|}-\{t\alpha_{\sigma_k}\}+\frac{1}{2} \right).
\end{equation}
We thus have
\begin{equation*}
\frac{s_{m0}}{s_{mt}} = \frac{\pi |e_kb^m| \left( 1+\mathcal{O}(b^{2m})\right)}{\pi |e_kb^m| \left( u_t+\mathcal{O}(|b|^{m/5})\right)} = \frac{1+\mathcal{O}(|b|^{m/5})}{u_t},
\end{equation*}
and moreover
\begin{align*}
\prod_{t=1}^{\eta}\left( 1-\frac{s_{m0}^2}{s_{mt}^2} \right) &= \prod_{t=1}^{\eta} \left( 1- \frac{1}{u_t^2} - \frac{\mathcal{O}(|b|^{m/5})}{u_t^2} \right) \\
&= \prod_{t=1}^{\eta} \left( 1-\frac{1}{u_t^2}\right) \prod_{t=1}^{\eta} \left( 1- \frac{\mathcal{O}(|b|^{m/5})}{u_t^2-1}\right).
\end{align*}
We look closer at the two products on the final line above. Since $|\xi_{\infty t}|<1/2$ and $|c_ke_k|<1$, we see from \eqref{eq:ut} that $u_t >1$ for all $1 \leq t \leq \eta$. This guarantees that both products are well-defined. Moreover, we see that $u_t$ behaves as $2t/|c_ke_k|$ for large $t$. Hence by comparison with $\sum 1/t^2 = \pi^2/6$, the sum $\sum 1/(u_t^2-1)$ converges, and it follows that $\sum \mathcal{O}(|b|^{m/5})/(u_t^2-1) = \mathcal{O}(|b|^{m/5})$. The latter sum is thus smaller than one, provided $m$ is sufficiently large, and again it follows from Lemma \ref{ProductSumLemma} that
\begin{equation}
\label{eq:smalletabound}
1 > \prod_{t=1}^{\eta} \left( 1- \frac{\mathcal{O}(|b|^{m/5})}{u_t^2-1}\right) \geq 1- \sum_{t=1}^{\eta} \frac{\mathcal{O}(|b|^{3m/5})}{u_t^2-1}  = 1- \mathcal{O}(|b|^{m/5}).
\end{equation}
For the second product we introduce the notation 
\begin{equation*}
U_j := \prod_{t=1}^j \left( 1- \frac{1}{u_t^2} \right).
\end{equation*} 
Since $u_t > 1$ for all $t$, the sequence $(U_j)_{j\geq 1}$ is monotonically decreasing and bounded below by zero. Thus, the limit $\lim_{j \rightarrow \infty} U_j$ exists.

By combining the estimates for $t>\eta$ and $t \leq \eta$, we now have 
\begin{equation*}
C_m = U_{\eta} \cdot \prod_{t=1}^{\eta} \left( 1- \frac{\mathcal{O}(|b|^{m/5})}{u_t^2-1}\right) \cdot \prod_{t=\eta+1}^{(q_n-1)/2} \left( 1- \frac{s_{m0}^2}{s_{mt}^2} \right).
\end{equation*}
Taking the limit of both sides as $m\rightarrow \infty$, and recalling \eqref{eq:largeetabound} and \eqref{eq:smalletabound}, we arrive at
\begin{equation}
\label{eq:limCmalmost}
\lim_{m \rightarrow \infty} C_m = \lim_{\eta\rightarrow \infty} U_{\eta} = \prod_{t=1}^{\infty} \left( 1- \frac{1}{u_t^2} \right),
\end{equation}
with $u_t$ given in \eqref{eq:ut}. This nearly completes the proof of Theorem \ref{thm:convCm}. Our claim, however, is that $\lim_{m\rightarrow \infty} C_m$ is strictly positive. This will follow from \eqref{eq:limCmalmost} and Lemma \ref{ProductSumLemma} if we can verify that 
\begin{equation}
\label{eq:smallenough}
\sum_{t=1}^{\infty} \frac{1}{u_t^2} < 1 .
\end{equation}

Let us first verify \eqref{eq:smallenough} for $\ell=1$. In this case, we have $k=0$ and $|c_0e_0|\leq1/\sqrt{5}$ by Lemma \ref{lem:absckek}. It follows that
\begin{equation*}
\sum_{t=1}^{\infty} \frac{1}{u_t^2} \leq \frac{1}{u_1^2}+ \sum_{t=2}^{\infty} \frac{1}{20(t-1)^2} < \frac{1}{4(\sqrt{5}-1/2)^2} + \frac{\pi^2}{120} < 1 .
\end{equation*} 
The case $\ell>1$ is more involved. However, it is easy to show that $p_{\ell}(\alpha_{\tau_k})-2b>0$ in this case, and accordingly it follows from Lemma \ref{lem:absckek} that 
\begin{equation*}
\frac{1}{|c_ke_k|} = \frac{q_{\ell+1}(\alpha_{\tau_k})+ p_{\ell}(\alpha_{\tau_k})-2b}{q_{\ell}(\alpha_{\tau_k})} 
\geq a_k + \frac{q_{\ell-1}(\alpha_{\tau_k})}{q_{\ell}(\alpha_{\tau_k})},
\end{equation*}
where we have also used the classical recursion formula for $q_{\ell}$. By Lemma~\ref{lem:qnquot} we get
\begin{equation*}
\frac{q_{\ell-1}(\alpha_{\tau_k})}{q_{\ell}(\alpha_{\tau_k})} = \frac{p_{\ell}(\alpha_{\sigma_0\tau_k})}{q_{\ell}(\alpha_{\sigma_0\tau_k})}= \frac{p_{\ell}(\alpha_{\sigma_k})}{q_{\ell}(\alpha_{\sigma_k})},
\end{equation*}
(recall from the proof of Lemma \ref{BasicPropC} that $\sigma_k=\sigma_0\tau_k$) and thus we have
\begin{equation}
\label{eq:u1est}
u_1 = 2 \left( \frac{1}{|c_ke_k|}-\alpha_{\sigma_k}+\frac{1}{2}\right) \geq 2\left( a_k + \frac{p_{\ell}(\alpha_{\sigma_k})}{q_{\ell}(\alpha_{\sigma_k})} - \alpha_{\sigma_k}+\frac{1}{2} \right).
\end{equation}
By the standard error estimate \eqref{eq:stderror} for continued fractions, we know that
\begin{equation*}
\left| \frac{p_{\ell}(\alpha_{\sigma_k})}{q_{\ell}(\alpha_{\sigma_k})} - \alpha_{\sigma_k} \right| < \frac{1}{q_{\ell+1}(\alpha_{\sigma_k})q_{\ell}(\alpha_{\sigma_k})} < \frac{1}{2}
\end{equation*}
when $\ell>1$, and inserting this in \eqref{eq:u1est} we find that $u_1\geq 2a_k \geq 2$. For all other terms in the sum $\sum 1/u_t^2$, the estimate $|c_ke_k|<1$ from Lemma \ref{lem:absckek} suffices. We get
\begin{equation*}
\sum_{t=1}^{\infty} \frac{1}{u_t^2} \leq \frac{1}{4} + \sum_{t=2}^{\infty} \frac{1}{4(t-1/2)^2} = \frac{1}{4} + \frac{\pi^2}{8}-1 < 0.49.
\end{equation*}
This verifies \eqref{eq:smallenough} for the case $\ell>1$. Thus, we conclude that $\lim_{m\rightarrow \infty} C_m >0$, and this completes the proof of Theorem \ref{thm:convCm}.
\end{proof} 


\section{Convergence of $B_m$\label{sec:Bm}}
The aim of this section is to verify the convergence of 
\begin{equation*}
B_m = \left| \prod_{t=1}^{q_{\ell m +k}-1} \frac{s_{mt}}{2 \sin (\pi t / q_{\ell m +k})} \right|
\end{equation*}
as $m \rightarrow \infty$. We will see that this requires greater efforts than verifying the convergence of $C_m$. In fact, what we will show is that $\log B_m$ converges to a finite limit, and accordingly $\lim_{m \rightarrow \infty} B_m$ exists and is strictly positive.

For the remainder of this section, let us again ease notation by writing $n=\ell m +k$. We begin by examining each term of the product $B_m$. Recalling the definition of $s_{mt}$ from \eqref{eq:sm}, we have
\begin{align*}
\frac{s_{mt}}{2\sin (\pi t/q_n)} &= \cos \pi |e_kb^m| \xi_{mt} - \cot (\pi t/q_n) \sin \pi |e_kb^m| \xi_{mt} \\
&= 1- 2\sin^2(\pi |e_kb^m| \xi_{mt}/2) - h_{mt},
\end{align*}
with $h_{mt}$ given in \eqref{eq:hm}. Taking $\beta_{mt}:= 2\sin^2(\pi |e_kb^m| \xi_{mt}/2)$, it is easily verified that $\beta_{m(q_n-t)} = \beta_{mt}$ for $t \in \{ 1, \ldots , q_n-1 \}$. Likewise, we recall from Lemma \ref{lem:propsm}\ref{it:propsm1} that $h_{m(q_n-t)}=h_{mt}$, and thus
\begin{equation*}
B_m = \prod_{t=1}^{q_n-1} (1-\beta_{mt}-h_{mt}) = \prod_{t=1}^{(q_n-1)/2} (1-\beta_{mt}-h_{mt})^2 . 
\end{equation*}
This shows that we need only consider $t \in \{ 1, \ldots , (q_n-1)/2 \}$. 

Let us now show that rather than analyzing $B_m$, we may choose to analyze the simpler product
\begin{equation}
\label{eq:Bstar}
B_m^* := \prod_{t=1}^{q_n-1} (1-h_{mt}) = \prod_{t=1}^{(q_n-1)/2} (1-h_{mt})^2.
\end{equation}
Taking logarithms, we get
\begin{equation}
\label{eq:logsBm}
\log (1-\beta_{mt}-h_{mt}) = \log (1-h_{mt})+ \log \left( 1- \frac{\beta_{mt}}{1-h_{mt}}\right).
\end{equation}
Our claim is that the latter term on the right hand side in \eqref{eq:logsBm} will not contribute significantly to the sum $\log B_m = \sum \log (1- \beta_{mt}-h_{mt})$. To see this, let us first estimate the size of $h_{mt}$ and $\beta_{mt}$. Considering only $t \in \{ 1, \ldots , (q_n-1)/2\}$, we use $\cot x < 1/x$ and $\sin x < x$ to obtain
\begin{equation*}
|h_{mt}| = \cot (\pi t /q_n) \sin \pi |e_kb^m| \xi_{mt} \leq \frac{q_n|e_kb^m|\xi_{mt}}{t}.
\end{equation*}
We recall from Lemmas \ref{QnEstimates} and \ref{lem:absckek} that $q_n|e_kb^m|=|c_ke_k| + \mathcal{O}(b^{2m}) <1$ for sufficiently large $m$. As $|\xi_{mt}|<1/2$, we thus get
\begin{equation}
\label{eq:hmbound}
|h_{mt}|< \frac{1}{2t} < \frac{1}{2},
\end{equation}
and it follows that $1-h_{mt}>1/2$. For $\beta_{mt}$, we have
\begin{equation*}
\beta_{mt} <  2 \left( \frac{\pi |e_kb^m|\xi_{mt}}{2} \right)^2 < \frac{\pi^2(e_kb^m)^2}{8},
\end{equation*}
and thus for sufficiently large values of $m$ we get $|\beta_{mt}/(1-h_{mt})|<1$ and
\begin{equation}
\label{eq:logzero}
\log \left( 1- \frac{\beta_{mt}}{1-h_{mt}}\right) = - \sum_{j=1}^{\infty} \frac{1}{j} \left( \frac{\beta_{mt}}{1-h_{mt}} \right)^j = \mathcal{O}(b^{2m}) .
\end{equation} 
Recalling that $q_n = \mathcal{O}(|b|^{-m})$, it now follows from \eqref{eq:logsBm} and \eqref{eq:logzero} that
\begin{equation*}
\left| \log B_m - \log B_m^* \right| = \left| 2 \sum_{t=1}^{(q_n-1)/2} \log \left( 1- \frac{\beta_{mt}}{1-h_{mt}} \right) \right| = \mathcal{O}(|b|^m),
\end{equation*}
and thus $\lim_{m \to \infty} \log B_m = \lim_{m \to \infty} \log B_m^*$. This confirms that we may choose to analyze $B_m^*$ in \eqref{eq:Bstar} rather than the original product $B_m$. 

Finally, we rewrite $\log B_m^*$ using its Taylor expansion as
\begin{equation}
\label{eq:Hmsplit}
\begin{aligned}
\log  B_m^* &= 2 \sum_{t=1}^{(q_n-1)/2} \log (1-h_{mt}) = -2 \sum_{t=1}^{(q_n-1)/2} \sum_{j=1}^{\infty} \frac{1}{j} h_{mt}^j \\
&=-2 \left( \sum_{t=1}^{(q_n-1)/2} h_{mt} + \sum_{t=1}^{(q_n-1)/2} \sum_{j=2}^{\infty} \frac{1}{j}h_{mt}^j \right) =: -2(H_m^{(1)}+H_m^{(2)}).
\end{aligned}
\end{equation}
We go on to study the behaviour of the two sums $H_m^{(1)}$ and $H_m^{(2)}$ separately in the following subsections.


\subsection{Convergence of $H_m^{(2)}$}
Let us first treat the sum
\begin{equation*}
H_m^{(2)} =  \sum_{t=1}^{(q_n-1)/2} \sum_{j=2}^{\infty} \frac{1}{j}h_{mt}^j .
\end{equation*}
It is an easy task to show that $H_m^{(2)}$ is bounded, but showing convergence requires greater efforts.

We begin by showing that terms where $t$ or $j$ is greater than $|b|^{-m/2}$ will not contribute significantly to $H_m^{(2)}$. Recall from \eqref{eq:hmbound} that $|h_{mt}|<1/(2t)$ for sufficiently large $m$, and thus for $u\geq 2$ we get
\begin{equation*}
\left|\sum_{j=u}^{\infty}\frac{1}{j}h^j_{mt}\right|<\sum_{j=u}^{\infty}|h^j_{mt}| =
		\frac{|h^u_{mt}|}{1-|h_{mt}|}<2\left(\frac{1}{2t}\right)^u.
\end{equation*}
Now let $u=\lfloor |b|^{-m/2} \rfloor$, and choose $m$ so that $2 \leq u \leq (q_n-1)/2$ and \eqref{eq:hmbound} holds.  We then have 
\begin{equation*}
	\left|\sum_{t=u+1}^{(q_n-1)/2}\sum_{j=2}^{\infty}\frac{1}{j}h^j_{mt}\right| <\sum_{t=u+1}^
		{(q_n-1)/2} 2\left(\frac{1}{2t}\right)^2 < \frac{1}{2}\sum_{t=u+1}^{\infty}\frac{1}{t^2} 
		< \frac{1}{2u}
\end{equation*}
and
\begin{equation*}
	\left|\sum_{t=1}^{u}\sum_{j=u+1}^{\infty}\frac{1}{j}h^j_{mt}\right| < \sum_{t=1}^{u} 
		2\left(\frac{1}{2t}\right)^{u+1} < \left(\frac{1}{2}\right)^u \sum_{t=1}^{\infty}
		\frac{1}{t^2} = \frac{\pi^2}{6 \cdot 2^{u}}.
\end{equation*}
Both of these sums are $\mathcal{O}(|b|^{m/2})$, and it follows that
\begin{align*}
	H_m^{(2)} &=  \sum_{t=1}^{u}\sum_{j=2}^{u} \frac{1}{j}h^j_{mt} + \sum_{t=1}^{u}\sum_{j=u+1}^{\infty}\frac{1}{j}h^j_{mt} + \sum_{t=u+1}^{(q_n-1)/2}\sum_{j=2}^{\infty}	\frac{1}{j}h^j_{mt}\\
	&= \sum_{t=1}^{u}\sum_{j=2}^{u} \frac{1}{j}h^j_{mt} + \mathcal{O}(|b|^{m/2}). 
\end{align*}
Thus, we have 
\begin{equation}
\label{eq:simSm}
H_m^{(2)} \sim \sum_{t=1}^{u}\sum_{j=2}^{u} \frac{1}{j}h^j_{mt},
\end{equation}
where we recall from Section \ref{sec:notation} that this notation means that the limit of $H_m^{(2)}$ equals that of its truncation $\sum_{t=1}^{u}\sum_{j=2}^{u} h^j_{mt}/j$ as $m\rightarrow \infty$.

Now let us see that
\begin{equation}
\label{eq:simhinfty}
 \sum_{t=1}^{u}\sum_{j=2}^{u} \frac{1}{j}h^j_{mt} \sim \sum_{t=1}^{u}\sum_{j=2}^{u} \frac{1}{j}h^j_{\infty t}, 
\end{equation}
where $h_{\infty t}$ is given in \eqref{eq:hinfty}. As we are considering only $j,t \leq u$, we have $jt \leq u^2 \leq |b|^{-m}$, and hence $jtb^{2m} \rightarrow 0$ as $m\rightarrow \infty$. From Lemma \ref{lem:propsm}\ref{it:propsm4}, we therefore get
\begin{equation*}
h_{mt}^j - h_{\infty t}^j = \left(h_{\infty t} + \mathcal{O}(tb^{2m}) \right)^j - h_{\infty t}^j = \mathcal{O}(jtb^{2m}),
\end{equation*}
and it follows that 
\begin{equation*}
\sum_{t=1}^{u}\sum_{j=2}^{u}\frac{1}{j}(h^j_{mt}-h^j_{\infty t}) = \sum_{t=1}^{u}\sum_{j=2}
^{u}\mathcal{O}(tb^{2m})=\mathcal{O}(u^3b^{2m})=\mathcal{O}(|b|^{m/2}).
\end{equation*}
This confirms \eqref{eq:simhinfty}.

Finally, by reusing the argument that led us to conclude that $H_m^{(2)}\sim \sum_{t=1}^{u}\sum_{j=2}^{u} h^j_{mt}/j$, we find that 
\begin{equation}
\label{eq:simhm}
\sum_{t=1}^{u}\sum_{j=2}^{u} \frac{1}{j}h^j_{\infty t} \sim \sum_{t=1}^{\infty} \sum_{j=2}^{\infty} \frac{1}{j}h^j_{\infty t}, 
\end{equation}
and recalling that $|h_{\infty t}| = |c_ke_k\xi_{\infty t}/t| < 1/(2t)<1/2$, we get
\begin{equation*}
\sum_{t=1}^{\infty} \sum_{j=2}^{\infty} \frac{1}{j} |h_{\infty t}|^j < \sum_{t=1}^{\infty} \frac{h_{\infty t}^2}{1-|h_{\infty t}|}<\infty.
\end{equation*}
Thus, the sum on the right hand side in \eqref{eq:simhm} is absolutely convergent. We denote its limit by $\Gamma_{\ell, k}^{(2)}$, and from \eqref{eq:simSm}--\eqref{eq:simhm} it follows that
\begin{equation}
\label{eq:limHm2}
\lim_{m\rightarrow \infty} H_m^{(2)} = \Gamma_{\ell, k}^{(2)} .
\end{equation}


\subsection{Convergence of $H_m^{(1)}$}
We are left with verifying the convergence of 
$$H_m^{(1)} = \sum_{t=1}^{(q_n-1)/2} h_{mt}.$$ 
This rather tedious task is performed in several steps. Eventually we will see that if $\lim_{m\rightarrow \infty} H_m^{(1)}$ exists, then it equals the limit of $\sum_{t=1}^{(q_n-1)/2} C_{mt}S_{mt}$, where $S_{mt}$ is a sum of sines and $C_{mt}$ is a cotangent difference. Careful estimates of $S_{mt}$ and $C_{mt}$ will reveal that the sum $\sum C_{mt} S_{mt}$ indeed converges.

Note first that we may return to standard summation notation at this point, as $h_{m(q_n/2)}=0$ if $q_n$ is even. Thus, we let $M_n:=\lfloor (q_n-1)/2 \rfloor$ and have
\begin{equation*}
H_m^{(1)} = \sum_{t=1}^{M_n} h_{mt} = \sum_{t=1}^{M_n}  \cot \left( \frac{\pi t}{q_n}\right) \sin \pi |e_kb^m| \xi_{m t},
\end{equation*}
regardless of whether $q_n$ is even or odd. 

Now let us see that 
\begin{equation}
\label{eq:HsimHstar}
H_m^{(1)} \sim H_m^* := \sum_{t=1}^{M_n} \cot \left( \frac{\pi t}{q_n}\right) \sin \pi |e_kb^m| \xi_{\infty t},
\end{equation}
where we again recall that $f(m) \sim g(m)$ means $\lim_{m\rightarrow \infty} f(m)/g(m) =1$. Using that $\sin x = x(1+\mathcal{O}(x^2))$ and Lemma \ref{lem:propsm}\ref{it:propsm3}, we get
\begin{align*}
H_m^{(1)}-H_m^* &= \sum_{t=1}^{M_n} \cot \left( \frac{\pi t}{q_n}\right) \pi |e_kb^m| (\xi_{mt}-\xi_{\infty t}) (1+\mathcal{O}(b^{2m})) \\
&= \sum_{t=1}^{M_n} \cot \left( \frac{\pi t}{q_n}\right) \pi |e_kb^m| \mathcal{O}(tb^{2m}).
\end{align*}
From the inequality $\cot x < 1/x$ it thus follows that
\begin{align*}
|H_m^{(1)}-H_m^*| &< \mathcal{O}(b^{2m})\sum_{t=1}^{M_n} q_n|e_kb^m| \\
&= \mathcal{O}(b^{2m}) \cdot M_n \left( |e_kc_k| + \mathcal{O}(b^{2m}) \right) = \mathcal{O}(|b|^m),
\end{align*}
where we have used Lemma \ref{QnEstimates} and the fact that $M_n< q_n = \mathcal{O}(|b|^{-m})$. This confirms \eqref{eq:HsimHstar}.

Finally, let us see that if $\lim_{m\rightarrow\infty} H_m^*$ exists, then it equals that of $\sum S_{mt} C_{mt}$ for a certain sum of sines $S_{mt}$ and cotangent difference $C_{mt}$. Using summation by parts, we may rewrite $H_m^*$ as
\begin{equation}
\label{SummationByParts}
\begin{aligned}
H_m^* &= \sum_{t=1}^{M_n-1}\left( \cot \left( \frac{\pi t}{q_n} \right) -\cot \left( \frac{\pi (t+1)}{q_n}\right)
		\right) \sum_{s=1}^t \sin \pi |e_{k}b^m| \xi_{\infty s} \\
	&+ \cot \left( \frac{\pi M_n}{q_n} \right) \sum_{s=1}^{M_n} \sin \pi |e_{k}b^m|
		\xi_{\infty s}. 
\end{aligned}
\end{equation}
Consider the second term on the right hand side in this equation. As $|\xi_{\infty s}|<1/2$ and $\sin x = x(1+\mathcal{O}(x^2))$, we have
\begin{equation*}
\left| \sum_{s=1}^{M_n} \sin \pi |e_kb^m| \xi_{\infty s} \right| < \frac{\pi}{2} q_n |e_kb^m| (1+\mathcal{O}(b^{2m}))= \frac{\pi}{2}|c_ke_k| (1+\mathcal{O}(b^{2m})),
\end{equation*}
where we have again used that $M_n<q_n=\mathcal{O}(b^{2m})$ and Lemma \ref{QnEstimates}. It follows that
\begin{equation*}
\left|\cot \left( \frac{\pi M_n}{q_n} \right) \sum_{s=1}^{M_n} \sin \pi |e_{k}b^m| \xi_{\infty s} \right| < \frac{\pi}{2}\left| c_ke_k \cot \left( \frac{\pi M_n}{q_n} \right) (1+\mathcal{O}(b^{2m}))\right|, 
\end{equation*}
and recalling that $M_n=\lfloor (q_n-1)/2 \rfloor$, it is clear that the cotangent term tends to zero as $m\rightarrow \infty$. It thus follows from \eqref{eq:HsimHstar} and \eqref{SummationByParts} that
\begin{equation}
\label{eq:simCS}
H_m^{(1)} \sim H_m^* \sim \sum_{t=1}^{M_n-1} C_{mt} S_{mt}
\end{equation}
where $C_{mt}:= \cot(\pi t/q_n) - \cot (\pi(t+1)/q_n)$ and $S_{mt}:=\sum_{s=1}^t \sin \pi |e_kb^m| \xi_{\infty s}$. 

\subsubsection{The cotangent difference $C_{mt}$ \label{subsub:cotangent}}
We establish two estimates for $C_{mt}$; one rather coarse bound and one more precise estimate. For ease of notation, let us write $\phi=\pi/q_n$. We then have
\begin{align*}
	0 < C_{mt} &= \frac{\sin((t+1)\phi)\cos(t\phi) - 
		\cos((t+1)\phi)\sin(t\phi)}{\sin(t\phi)\sin((t+1)\phi)}\\
	&= \frac{\sin(\phi)}{\sin(t\phi)\sin((t+1)\phi)}.
\end{align*}
Note that when $t < M_n$, we have $(t+1)\phi < \pi/2$, and thus by $2x/\pi < \sin x < x$, we obtain
\begin{equation}
\label{eq:coarseCm}
0<C_{mt} < \frac{\pi q_n}{4t(t+1)} < \frac{\pi q_n}{4t^2}.
\end{equation}
This is our coarse bound for $C_{mt}$.

For $t\phi <1$, or equivalently $t<q_n/\pi$, we have the finer estimate
\begin{align*}
C_{mt} & = \frac{\phi(1+\mathcal{O}(\phi^2))}{t\phi(1+\mathcal{O}(t^2\phi^2))
		(t+1)\phi(1+\mathcal{O}(t^2\phi^2))}\\
	& = \frac{q_n}{\pi t(t+1)}\left(1+ \mathcal{O} \left( \frac{t^2}{q_n^2}\right)\right) = \frac{q_n}{\pi t(t+1)}\left(1+ \mathcal{O} \left( t^2 b^{2m} \right)\right),
\end{align*}
where we have used that $1/q_n = \mathcal{O}(|b|^{m})$. Combining this estimate with 
\begin{equation*}
S_{mt} = \pi |e_kb^m|  (1+\mathcal{O}(b^{2m})) \sum_{s=1}^{t} \xi_{\infty s} ,
\end{equation*}
and using Lemma \ref{QnEstimates}, we get
\begin{equation}
\label{eq:estCmSm}
C_{mt}S_{mt} = \frac{|c_ke_k|}{t(t+1)} \left( 1+\mathcal{O}(t^2b^{2m}) \right) \sum_{s=1}^t \xi_{\infty s}. 
\end{equation}

\subsubsection{The sum of sines $S_{mt}$}
Now let us find an appropriate bound on $S_{mt}$ in terms of $m$ and $t$. As illustrated by Mestel and Verschueren in \cite[Figure 7.1]{MV15}, this sum appears to grow slowly with increasing values of $t$, at least for the specal case of $\alpha=\omega$ the golden mean. As demonstrated by the next lemma, this is also true for the general case where $\ell\geq 1$.
\begin{lem}
\label{lem:Smtbound}
For $t \in \{ 1,2, \ldots , q_n-1 \}$, the sum 
\begin{equation*}
S_{mt} = \sum_{s=1}^t \sin \pi |e_kb^m| \xi_{\infty s}
\end{equation*}
satisfies $|S_{mt}| = \mathcal{O}(|b|^m\log t)$. 
\end{lem}
For proving Lemma \ref{lem:Smtbound}, we will need the following result.
\begin{lem}
\label{lem:sumfrac}
Let $p/q$ be a convergent of any real $\alpha$. Then for any $\theta \in \RR$ and $v \in \NN$, we have
	$$\left| \sum_{i=1}^{vq} \{\theta +i\alpha\} -\frac{1}{2} \right|<\frac{3v}{2}. $$
\end{lem}
\begin{proof}
The proof for $v=1$ is given in \cite[Lemma 7.2]{MV15}. For $v\geq 2$ it follows that
\begin{align*}
		\left| \sum_{i=1}^{vq} \{\theta +i\alpha\} -\frac{1}{2} \right| &= \left|\sum_{j=0}
		^{v-1}	\sum_{u=1}^{q} \{\theta +(jq+u)\alpha\} -\frac{1}{2} \right|\\
		&\leq \sum_{j=0}^{v-1} \left| \sum_{u=1}^{q} \{(\theta+jq\alpha) +u\alpha\} -\frac{1}{2} \right| < \sum_{j=0}^{v-1} \frac{3}{2}=\frac{3v}{2}.
	\end{align*}
\end{proof}

\begin{proof}[Proof of Lemma \ref{lem:Smtbound}]
Recall from Lemma \ref{OstrowskiRepr} that there exist unique integers $z, v_1, \ldots , v_z \in \NN$ such that
\begin{equation*}
t=\sum_{s=1}^z v_sq_s .
\end{equation*}
We will use this representation of $t$ to split the sum $S_{mt}$ into chunks of length $v_sq_s$ as follows. Let us introduce the notation $t_z=0$ and $t_s=\sum_{u=s+1}^z v_uq_u$. Moreover, we define
\begin{equation*}
\xi_{\infty r}(\theta) := \left\{ \theta +r\alpha_{\sigma_k} - \frac{1}{2} \right\}.
\end{equation*}
Note that our $\xi_{\infty r}$ defined in \eqref{eq:xinfty} is then precisely $\xi_{\infty r}(0)$. With this generalized $\xi_{\infty r}(\theta)$ introduced, we may rewrite $S_{mt}$ as
\begin{align*}
S_{mt} &=  \sum_{r=1}^{v_zq_z}\sin \pi |e_{k}b^m|\xi_{\infty r}(0) + \sum_{r=1}^{v_{z-1} q_{z-1}}\sin \pi |e_{k}b^m|\xi_{\infty r}(v_z q_z \alpha_{\sigma_k}) \\
	   &+ \sum_{r=1}^{v_{z-2}q_{z-2}}\sin \pi |e_{k}b^m|\xi_{\infty r}((v_zq_z+ v_{z-1}q_{z-1})\alpha_{\sigma_k}) + \ldots \\ 
	   &= \sum_{s=1}^{z} \sum_{r=1}^{v_sq_s} \sin \pi |e_{k}b^m| \xi_{\infty r}(t_s \alpha_{\sigma_k}).
\end{align*}
Thus, if we also introduce the generalized notation
\begin{equation*}
S_{mt}(\theta) = \sum_{r=1}^t \sin \pi |e_kb^m| \xi_{\infty r}(\theta),
\end{equation*}
then we can express $S_{mt}$ as
\begin{equation}
\label{eq:SmsumSm}
S_{mt} =S_{mt}(0)= \sum_{s=1}^z S_{m(v_sq_s)}(t_s\alpha_{\sigma_k}) . 
\end{equation}

Finally we use Lemma \ref{lem:sumfrac} to bound the terms in the sum \eqref{eq:SmsumSm}. Using the estimate $\sin x = x(1+\mathcal{O}(x^2))$, we get
\begin{align*}
\left| S_{m(v_sq_s)}(t_s\alpha_{\sigma_k}) \right| &= \left| \pi e_kb^m(1+\mathcal{O}(b^{2m})) \sum_{r=1}^{v_sq_s} \xi_{\infty r} (t_s\alpha_{\sigma_k})\right| \\
&= \pi |e_kb^m| (1+\mathcal{O}(b^{2m})) \left| \sum_{r=1}^{v_sq_s} \left\{ t_s\alpha_{\sigma_k} + r\alpha_{\sigma_k} \right\} - \frac{1}{2} \right|\\
&\leq \frac{3}{2} \pi v_s |e_kb^m| (1+\mathcal{O}(b^{2m})).  
\end{align*}
Thus, we have
\begin{equation}
\label{eq:Smmedbound}
|S_{mt}| \leq \sum_{s=1}^z |S_{m(v_sq_s)}(t_s\alpha_{\sigma_k})| \leq \frac{3}{2}\pi |e_kb^m|(1+\mathcal{O}(b^{2m}))\sum_{s=1}^z v_s . 
\end{equation}
Recalling from Lemma \ref{OstrowskiRepr} that $v_s \leq \max \{ a_1, \ldots , a_{\ell} \}$ for all $s$, we have
\begin{equation*}
\sum_{s=1}^{z} v_s \leq z \cdot \max_{1\leq j \leq \ell} a_j = \mathcal{O}(z) = \mathcal{O}(\log t) ,
\end{equation*}
and combined with \eqref{eq:Smmedbound} this yields $|S_{mt}| = \mathcal{O}(|b|^{m} \log t)$.
\end{proof}

We are now equipped to prove the convergence of $H_m^{(1)}$, or equivalently the convergence of $\sum_{t=1}^{M_n-1} C_{mt}S_{mt}$ in \eqref{eq:simCS}. From \eqref{eq:coarseCm} and Lemma \ref{lem:Smtbound} it follows that
\begin{equation}
\label{eq:logtt2}
|C_{mt}S_{mt}| \leq \frac{\pi q_n}{4t^2} \cdot \mathcal{O}(|b|^m \log t) = \mathcal{O}\left( \frac{\log t}{t^2}\right),
\end{equation}
where we have also used $q_n|b|^m = c_k + \mathcal{O}(b^{2m})$ from Lemma \ref{QnEstimates}. This implies that there exists a constant $K>0$ such that
\begin{equation}
\label{eq:sumboundK}
\sum_{t=1}^{M_n-1} |C_{mt} S_{mt}| \leq K.
\end{equation}
It is not clear that the sequence $(\sum_{t=1}^{M_n-1} C_{mt}S_{mt})_{m\geq 1}$ is monotone, so the bound \eqref{eq:sumboundK} alone does not prove convergence. But let us now compare this sequence to a closely related, absolutely convergent sum.

Let $u=\lfloor |b|^{-m/2} \rfloor$, and choose $m$ sufficiently large for $u<q_n/\pi<M_n-1$. We can then write 
\begin{equation}
\label{eq:split}
\sum_{t=1}^{M_n-1} C_{mt}S_{mt} = \sum_{t=1}^{u} C_{mt}S_{mt} + \sum_{t=u+1}^{M_n-1} C_{mt}S_{mt}.
\end{equation}
It follows from \eqref{eq:logtt2} that
\begin{equation}
\label{eq:tail}
\sum_{t=u+1}^{M_n-1} C_{mt}S_{mt} = \mathcal{O}\left( \frac{\log u}{u}\right). 
\end{equation}
For the first sum on the right hand side in \eqref{eq:split}, we use the finer estimate \eqref{eq:estCmSm} from Section \ref{subsub:cotangent} to obtain
\begin{equation}
\label{eq:head}
\sum_{t=1}^{u} C_{mt}S_{mt} = (1+\mathcal{O}(|b|^m)) \sum_{t=1}^u \frac{|c_ke_k|}{t(t+1)} \sum_{s=1}^t \xi_{\infty s} .
\end{equation}
It follows from \eqref{eq:sumboundK} that both sides in \eqref{eq:head} are bounded by $K$ in absolute value. Thus, the series $\sum_{t=1}^u |c_ke_k|/(t(t+1)) \sum_{s=1}^t \xi_{\infty s}$ is absolutely convergent, and converges to some real number $\Gamma_{\ell, k}^{(1)}$ as $u\rightarrow \infty$. Finally, by combining \eqref{eq:split}--\eqref{eq:head}, it follows that
\begin{equation}
\label{eq:limHm1}
 \lim_{m\rightarrow \infty} H_m^{(1)} = \lim_{m \rightarrow \infty} \sum_{t=1}^{M_n-1} C_{mt}S_{mt} =  \Gamma_{\ell, k}^{(1)}.
\end{equation}

\subsection{Conclusion \label{subsec:Bmconc}}
Combining \eqref{eq:Hmsplit}, \eqref{eq:limHm2} and \eqref{eq:limHm1}, we finally arrive at
\begin{equation*}
\lim_{m \rightarrow \infty} \log B_m^* = -2 \left( \Gamma_{\ell, k}^{(1)} + \Gamma_{\ell, k}^{(2)} \right).
\end{equation*}
Recalling that $\log B_m^* \sim \log B_m$, it follows that $\log B_m$ converges to a finite limit, and accordingly the product $B_m$ in \eqref{eq:Bm} converges to a strictly positive number. 


\section{Proof of Theorem \ref{MainThm} \label{sec:proof}}
The proof of Theorem \ref{MainThm} is essentially completed. Nevertheless, we include a brief summary. Theorem \ref{MainThm} states that if $\alpha=[0;\overline{a_1, \ldots , a_{\ell}}]$ is an irrational with a periodic continued fraction expansion, then there are positive constants $C_0, \ldots , C_{\ell-1}$ such that
\begin{equation}
\label{eq:conc}
\lim_{m\rightarrow \infty} Q_{\ell m +k}(\alpha) = \prod_{r=1}^{q_{\ell m +k}} |2\sin \pi r \alpha| = C_k
\end{equation}
for each $k=0, 1, 2, \ldots , \ell -1$. By Lemma \ref{lem:decomp}, the product $Q_{\ell m +k}(\alpha)$ for fixed $k$ can be decomposed as
\begin{equation}
\label{eq:decompconc}
Q_{\ell m+k}(\alpha)=A_mB_mC_m, 
\end{equation}
where $A_m$, $B_m$ and $C_m$ are defined in \eqref{eq:Am}--\eqref{eq:Cm}. We have seen in Section \ref{sec:Amconv} that
$$\lim_{m\rightarrow \infty} A_m = 2\pi|c_ke_k| > 0.$$
Moreover, by Theorem \ref{thm:convCm} we have $\lim_{m\rightarrow \infty} C_m>0$, and finally we have seen in Section \ref{subsec:Bmconc} that also $\lim_{m\rightarrow \infty} B_m>0$. It thus follows from \eqref{eq:decompconc} that \eqref{eq:conc} holds for some $C_k>0$.

\subsection{Proof of Corollary \ref{Cor:preperiod}}

We only sketch the proof of Corollary \ref{Cor:preperiod}, as it largely follows that of Theorem~\ref{MainThm}. Let $\beta=[a_0;a_1,\ldots,a_h,\overline{a_{h+1},\ldots,a_{h+\ell}}]$ and $\alpha=[0;\overline{a_{h+1},\ldots,a_{h+\ell}}]$. It is an easy exercise to verify the identity 
	\begin{equation}\label{SplitOfPeriod}
		q_{h+u}(\beta)=q_{h+1}(\beta)q_{u}(\alpha) + q_{h}(\beta)p_u(\alpha)
	\end{equation}
for all $u\geq 0$. By combining \eqref{SplitOfPeriod} with Theorem \ref{thm:lehmer} and Lemmas \ref{QnEstimates} and \ref{BasicPropQn} for the purely periodic case, one can establish the closed form 
	\begin{equation}\label{ClosedFormPrePeriod}
		q_{h+\ell m +k}(\beta)=\gamma_1^{(m)}q_{h+\ell+k}(\beta)L_m + (-1)^{\ell-1}\gamma_2^{(m)}q_{h+k}(\beta)L_{m-1} ,
	\end{equation}
where $L_m=L_m(c(\alpha)^2,(-1)^{l-1})$ is the Lehmer sequence and $\gamma_1^{(m)}$ and $\gamma_2^{(m)}$ are defined as in Theorem \ref{thm:lehmer}. Moreover, one can find constants $c_{h,k}$ and $e_{h,k}$ (independent of $m$) such that 
\begin{equation}\label{eq:qbetabm}
q_{h+\ell m +k}(\beta)|b|^m = c_{h,k} + \mathcal{O}(b^{2m})
\end{equation}
and
\begin{equation}\label{eq:qbetab}
q_{h+\ell m +k}(\beta)\beta = p_{h+\ell m+k}(\beta) + e_{h,k}b^m,
\end{equation}
with $b=b(\alpha)$ defined in \eqref{eq:b}. Note that \eqref{eq:qbetabm} is essentially Lemma~ \ref{QnEstimates} for the irrational $\beta$, and similarly \eqref{eq:qbetab} corresponds to Lemma \ref{BasicPropQn}. Further calculations verify that 
\begin{equation}
\frac{q_{h+\ell m + k-1}(\beta)}{q_{h+\ell m + k}(\beta)} = \alpha_{\sigma_k} + \mathcal{O}(b^{2m}) \label{XiPrePeriod},
\end{equation}
which is basically Lemma \ref{lem:qnquot} for $\beta$. Thus, we have all tools needed to prove that the limit 
$$\lim_{m\rightarrow \infty} Q_{h+\ell m +k}(\beta)$$
indeed exists for each $k\in \{ 0,1, \ldots \ell-1 \}$. Finally, it turns out that the product $|c_{h,k}e_{h,k}|$ is independent of $h$, that is 
\begin{equation}\label{eq:indeph}
|c_{h,k}e_{h,k}|=|c_ke_k|,
\end{equation}
with $c_k$ and $e_k$ given in \eqref{eq:ck} and \eqref{eq:conste}.
By carefully examining the proof of Theorem \ref{MainThm}, it is clear that \eqref{eq:indeph} guarantees that
$$\lim_{m\rightarrow \infty} Q_{h+\ell m +k}(\beta)= \lim_{m\rightarrow \infty} Q_{\ell m +k}(\alpha),$$
and this completes the proof of Corollary \ref{Cor:preperiod}.


\section{Concluding remarks\label{sec:remarks}}

Let us finally return to the general sequence of sine products 
$$P_n(\alpha) = \prod_{r=1}^n |2 \sin \pi r \alpha |.$$
As a consequence of Theorem \ref{thm:mvgolden}, Mestel and Verschueren show in \cite{MV15} that one can establish polynomial bounds on $P_n(\alpha)$ when $\alpha=\omega$ is the golden mean. Specifically, they show that in this case there exist constants $K_1 \leq 0 < 1 \leq K_2$ such that 
\begin{equation}
\label{eq:polbounds}
n^{K_1} \leq P_n(\alpha) \leq n^{K_2} ,
\end{equation}
for all $n \in \NN$. It is worth mentioning that this is not a new result; in a paper from 1999, Lubinsky studies the product $P_n(\alpha)$ in the language of $q$-series \cite{LUB99}. In particular, he proves that \eqref{eq:polbounds} holds whenever $\alpha=[0;a_1, a_2, \ldots]$ has bounded continued fraction coefficients \cite[Theorem 1.3 II]{LUB99}. His result covers not only the golden mean, but in fact all quadratic irrationals $\alpha$. Accordingly, we do not attempt to deduce \eqref{eq:polbounds} for quadratic irrationals $\alpha$ as a consequence of Theorem \ref{MainThm}.

In the same paper, it is conjectured by Lubinsky that 
\begin{equation}
\label{eq:lub}
\liminf_{n \rightarrow \infty} P_n(\alpha) = 0
\end{equation}
for all irrationals $\alpha$ \cite[p.~220]{LUB99}. Lubinsky himself verifies \eqref{eq:lub} in the presence of \emph{unbounded} continued fraction coefficients of $\alpha=[0;a_1, a_2,...]$, and goes on to say that he believes it must be true in general. Theorems \ref{thm:mvgolden} and \ref{MainThm}, however, suggest otherwise. Numerical calculations indicate that it is precisely along the sequence of best approximation denominators $(q_n)_{n\geq 0}$ of $\alpha$ that $P_n(\alpha)$ takes on its minimum values, in the sense that
\begin{equation}
\label{eq:greater}
P_j(\alpha) \geq P_{q_n}(\alpha) \quad \text{ for } q_{n-1}<j<q_n .
\end{equation}
If \eqref{eq:greater} indeed holds, then Theorems \ref{thm:mvgolden} and \ref{MainThm} imply that Lubinsky's conjecture cannot hold for quadratic irrationals $\alpha$. Unfortunately, we have not succeeded in finding or establishing a rigorous proof of \eqref{eq:greater}, and accordingly Lubinsky's conjecture is (to the best of our knowledge) still open.


\Addresses


\begin{thebibliography}{10}
\bibitem{bell} J.~Bell, \emph{Estimates for the norms of products of sines and cosines}, J. Math. Anal. Appl. \textbf{405}, (2013).
\bibitem{buslaev} V.~I.~Buslaev, \emph{Convergence of the Rogers-Ramanujan continued fraction}, Sb. Math. \textbf{194} (2003), 833--856.
\bibitem{borgain} J.~Borgain and M.-C.~Chang, \emph{On a paper of Erd\H{o}s and Szekeres}, J. Anal. Math. (to appear).
\bibitem{driver} K.~A.~Driver, D.~S.~Lubinsky, G.~Petruska and P.~Sarnak, \emph{Irregular distribution of $\{n\beta\}$, $n=1,2,3,\ldots$, quadrature of singular integrands, and curious basic hypergeometric series}, Indag. Math. (N.S.) \textbf{2} (1991), 469--481.
\bibitem{erdos} P.~Erd\H{o}s and G.~Szekeres, \emph{On the product $\prod_{k=1}^n (1-z^{a_k})$}, Acad. Serbe Sci. Publ. Inst. Math. \textbf{13} (1959), 29--34.
\bibitem{freiman} G.~Freiman and H.~Halberstam, \emph{On a product of sines}, Acta Arith. \textbf{49} (1988), 377--385.
\bibitem{KT11} O.~Knill and F.~Tangerman, \emph{Self-similarity and growth in Birkhoff sums for the golden rotation}, Nonlinearity \textbf{24} (2011), 3115--3127.
\bibitem{KNUnidist} L.~Kuipers and H.~Niederreiter, \emph{Uniform distribution of sequences}, Wiley, New York, 1974. 
\bibitem{Le30} D. H. Lehmer, \emph{An Extended Theory of Lucas’ Functions}, The Annals of Mathematics, 2nd Ser., Vol. 31, No. 3 (1930), 419–448.
\bibitem{LUB99} D.~S~Lubinsky, \emph{The size of $(q;q)_n$ for $q$ on the unit circle}, J. Number Theory \textbf{76} (1999), 217--247.
\bibitem{mullin} R.~C.~Mullin, \emph{Some Trigonometric Products}, Amer. Math. Monthly, \textbf{69}(3) (1962), 217--218.
\bibitem{Perron} O.~Perron, \emph{Die Lehre von den Kettenbr\"uchen}, Stuttgart, Teubner, 1913 (German).
\bibitem{petruska} G.~Petruska, \emph{On the radius of convergence of $q$-series}, Indag. Math. \textbf{3} (1992), 353--364.
\bibitem{sudler} C.~Sudler Jr., \emph{An estimate for a restricted partition function}, Quart. J. Math. Oxford Ser. \textbf{15} (1964), 1--10.
\bibitem{MV15} P.~Verschueren and B.~Mestel, \emph{Growth of the Sudler product of sines at the golden rotation number}, J. Math. Anal. Appl. \textbf{433} (2016), 200--226.
\end{thebibliography}
\end{document}